\documentclass[12pt]{amsart}
\usepackage{amssymb}
\usepackage{amsthm}
\usepackage{enumerate}
\usepackage{graphicx}

\usepackage[latin1]{inputenc}

\newcommand{\NN}{\mathbb{N}}
\newcommand{\ZZ}{\mathbb{Z}}

\newcommand{\RR}{\mathbb{R}}
\newcommand{\CC}{\mathbb{C}}

\newcommand{\HH}{\mathbb{H}}

\newcommand{\teich}{\mathcal T}
\newcommand{\proj}{\mathcal P}

\newcommand{\Mod}{\mathrm{Mod}}
\newcommand{\gr}{\mathrm{gr}}
\newcommand{\im}{\mathrm{im}}
\newcommand{\Gr}{\mathrm{Gr}}
\newcommand{\arccosh}{\mathrm{arccosh}}
\newcommand{\union}{\cup}

\theoremstyle{definition}
\newtheorem{defi}{Definition}[section]

\theoremstyle{remark}

\theoremstyle{plain}
\newtheorem{lemma}{Lemma}[section]
\newtheorem{prop}[lemma]{Proposition}
\newtheorem{theorem}{Theorem}[section]
\newtheorem*{theorem*}{Theorem}
\newtheorem{cor}[theorem]{Corollary}

\begin{document}

\author{Sebastian W. Hensel}
\title{Iterated grafting and holonomy lifts of Teichmüller space} 
\maketitle

\section{Introduction}
\label{sec:intro}
Let $S$ be a closed oriented surface of genus $g \geq 2$. A \textit{marked complex structure} on $S$
is a pair $(X, f)$, where $X$ is a Riemann surface and $f$ is a \textit{marking}, that is a
orientation preserving homeomorphism $f:S \to X$.
The \textit{Teichmüller space} $\teich(S)$ of $S$ is the space of marked complex structures up to
isotopy. This set is made into a complete metric space by the \textit{Teichmüller metric} $d_\teich$. The uniformization theorem allows us to identify $\teich(S)$ with the space of marked hyperbolic structures on $S$.

One can also consider the (finer) notion of \textit{complex projective surfaces}. A complex projective surface $Z$ is given by a topological surface $S$, together with an atlas for $S$ whose charts have values in the Riemann sphere $\CC P^1$, and such that the chart transition maps are (locally) restrictions of Möbius transformations. 
Following the definition of Teichmüller space, one considers the space $\proj(S)$ of marked complex projective structures up to isotopy.
Since Möbius transformations are biholomorphic, any complex projective surface also is a Riemann surface, and thus we obtain a (forgetful) projection
$$\pi: \proj(S) \to \teich(S)$$
On the other hand, by the uniformization theorem any Riemann surface $X$ of genus $g \geq 2$ can be written as $X = \HH^2/\Gamma$, where $\HH^2$ is the upper half plane and $\Gamma$ is a discrete subgroup of $\mathrm{PSL}_2(\RR)$. The covering projection $\HH^2 \to \HH^2/\Gamma$ induces a natural projective structure on $X$ which we call the \textit{Fuchsian projective structure}.
In other words, we obtain a map
$$s: \teich(S) \to \proj(S)$$
such that $\pi\circ s$ is the identity. Thus $\teich_0(S) := s(\teich(S)) \subset \proj(S)$ is a copy of Teichmüller space in the space of projective structures.

Associated to a complex projective surface $Z$ is its \textit{holonomy representation} $\rho_Z:\pi_1(S) \to \mathrm{PSL}_2(\CC)$ (for details, see for example \cite[section 3.5]{Thurston:cr} or \cite{McMullen:1998mz}).
We say a projective surface $Z$ has \textit{Fuchsian holonomy} if $\rho_Z:\pi_1(S) \to \mathrm{PSL}_2(\CC)$ is an isomorphism onto (a conjugate of) a Fuchsian group.
A basic example of projective surfaces with Fuchsian holonomy is given by the Fuchsian projective structures defined above -- the holonomy group of $s(\HH^2/\Gamma)$ is just $\Gamma$.

One can describe all projective surfaces having Fuchsian holonomy explicitly using \textit{grafting} along weighted geodesic multicurves. 
Informally, to obtain the grafting $\Gr_{t\gamma}X$ of $X$ along the weighted geodesic $t\gamma$ one inserts a flat euclidean cylinder of height $t$ at $\gamma$ (cf. figure \ref{fig:step}) to obtain a new projective structure from the Fuchsian one (see \cite{Tanigawa:1995fk} or \cite{McMullen:1998mz} for details).
For weighted multicurves $t_1\gamma_1 + \ldots + t_n\gamma_n$ one cuts $X$ at all $\gamma_i$ and glues in the flat cylinders $\gamma_i\times[0,t_i]$ at the respective boundary components.

Goldman's theorem (\cite{Goldman:1987oe}) states that a projective structure $Z \in \proj(S)$ has Fuchsian holonomy if and only if it is of the form $Z = \Gr_{\lambda}X$ for some hyperbolic surface $X$ and an 
\textit{integral lamination} $\lambda$, that is a weighted multicurve $\lambda= 2\pi n_1\gamma_1 + \ldots + 2\pi n_r\gamma_r$ with $n_i \in \NN$. 
Thus the projective structures with Fuchsian holonomy -- the ``\textit{holonomy lifts}'' of Teichmüller space -- are given by
$$\teich_\lambda(S) := \Gr_{\lambda}(\teich(S))$$
for integral laminations $\lambda$. 
A general classification theorem for projective structures (see \cite{Kamishima:1992lr}) implies that the map ${\mathcal{IL}}(S)\times\teich(S) \to \proj(S)$ given by grafting is injective, where $\mathcal{IL}(S)$ is the space of integral laminations. 
Thus the holonomy lifts of Teichmüller space are disjoint slices in $\proj(S)$. From the same classification theorem it also follows that for any integral $\lambda$, $\Gr_\lambda:\teich(S) \to \proj(S)$ is a homeomorphism onto its image, and thus the slices $\teich_\lambda(S)$ are copies of Teichmüller space.

As the \textit{conformal grafting map} $\gr_\lambda = \pi\circ\Gr_\lambda: \teich(S) \to \teich(S)$ also is a homeomorphism for any integral lamination $\lambda$ by a result of Tanigawa (\cite{Tanigawa:1995fk}, see also \cite{McMullen:1998mz} and \cite{Scannell:2002yq})
a holonomy lift $\teich_\lambda(S)$ of Teichmüller space inherits from $\teich(S)$ two natural parametrizations:
On the one hand, we have the grafting coordinates $Z = \Gr_\lambda X \mapsto X$, and on the other hand there are the conformal coordinates $Z \mapsto \pi(Z)$. 
To understand the relation of these two coordinate systems, one has to study the conformal grafting map.

In this paper we consider the lifts of Teichmüller geodesics into the slices $\teich_\lambda(S)$. For a hyperbolic surface $X \in \teich(S)$ and a simple closed curve $\gamma$ on $S$, let $l_X(\gamma)$ be the length of the hyperbolic geodesic on $X$ in the free homotopy class of $\gamma$. 
We say that a (weighted) multicurve $\lambda = a_1\gamma_1 + \ldots a_r\gamma_r$ has length less than $\epsilon$ on $X$ if $l_X(\gamma_i) < \epsilon$ for all $i=1, \ldots, r$.We obtain the following
\begin{theorem}
  There is a number $\epsilon>0$ such that the following holds. Let $\delta>0$ and an integral lamination $\lambda$ be given. Consider the set $\,\mathcal{U} \subset \teich(S)$ of all hyperbolic surfaces on which $\lambda$ has length less than $\epsilon$ and each simple closed curve disjoint from $\lambda$ has length at least $\delta$.

  Then there is a number $r>0$, such that for each $X \in \mathcal{U}$ and $n \in \NN$ the holonomy lift 
    $$g_n(s) = \gr_{n\lambda}\left( \rho_{\lambda,X}(s) \right)$$
of the Teichmüller geodesic $\rho_{\lambda,X}$ through $X$ in direction $\lambda$ is contained in the $r$-tube around the geodesic $\rho_{\lambda, X}$.
\end{theorem}
Thus, the conformal grafting map (or ``holonomy lift map'') is well behaved on Teichmüller geodesics once the curves are
short: grafting in the direction of the ray basically moves forward on the geodesic.

The theorem is proved by studying the behaviour of the holonomy lift map on \textit{grafting rays}.  A grafting ray is a curve of the form $t \mapsto \gr_{t\lambda}X$ in Teichmüller space. These curves share many properties with Teichmüller geodesics.
For example, for any two points $X,Y$ in Teichmüller space, there is a unique grafting ray from $X$ to $Y$ (this follows from a far more general result in \cite{Dumas:2007fj}). Teichmüller geodesics are contained in Teichmüller disks, grafting rays also naturally define holomorphic disks in $\teich(S)$ (\textit{complex earthquake disks}, cf. \cite{McMullen:1998mz}).
Furthermore, grafting rays have the same asymptotic behaviour as Teichmüller geodesics. 
Diaz and Kim \cite{KD07} have shown that for any $X \in \teich(S)$ and integral lamination $\lambda$, the grafting ray $\gr_{t\lambda}X$ is contained in an $L$-tube around the Teichmüller geodesic ray from $X$ in direction $\lambda$, where $L$ depends on $X$.

We show the following theorem about holonomy lifts of grafting rays.
\begin{theorem}
  There is a number $\epsilon>0$ such that the following holds.
  Let $X$ be a hyperbolic surface and $\lambda$ be an integral lamination of length less than $\epsilon$ on $X$.
  \begin{enumerate}[i)]
  \item There is an $r>0$, such that for each $n$ the holonomy lift 
    $$g_n(s) = \gr_{n\lambda}\left( \gr_{s\lambda}X \right)$$
    is contained in the $r$-tube around the grafting ray $s \mapsto \gr_{s\lambda}X$.
   \item There is an $R>0$ such that the following holds. Let $\eta$ 
     be an short integral lamination, disjoint from $\lambda$. Then the holonomy lifts
     $$g_\eta(s) = \gr_{\eta}\left(\gr_{s\lambda}X\right)$$
     are contained in the $R$-tube around the grafting ray $\gr_{s\lambda}(\gr_{\eta}X)$.
  \end{enumerate}
\end{theorem}
The difficult part of this theorem is to establish that the constants $r$ and $R$ do not depend on n (or $\eta$).
The translation length of the grafting map $X \mapsto \gr_{n\cdot\lambda}X$ can be estimated from the length of $\lambda$ on $X$ and the weights in $\lambda$, so each individual holonomy lift will be contained in a suitable tube around the grafting ray. However, as the translation length is unbounded in $n$ it is a priori not clear that \textit{all} holonomy lifts lie in a single tube.

We also consider grafting rays through holonomy lifts of some starting point $X$. 
\begin{theorem}
   Let $X$ be a hyperbolic surface and $\gamma$ a simple closed geodesic on $X$. Consider the grafting rays
  $$c_{n,m}(t) = \gr_{t\gamma}(\gr_{2\pi m\gamma}^nX).$$
  For large values of $n$, the $c_{n,m}$ accumulate exponentially fast
  $$d_\teich(c_{n+1,m}(t), c_{n,m}(2\pi m + a_{n,m}t)) \leq C\cdot q^n$$
  for some $0<q<1, a_{n,m} > 1$ and a constant $C$ depending on $X$. In particular, these rays accumulate in the Hausdorff topology on Teichmüller space.
\end{theorem}
To understand the behaviour of the conformal grafting map on grafting rays, one needs to
understand how grafting behaves under iteration. 
To this end, note that grafting does not form a flow, i.e. $\gr_{t\lambda}\gr_{s\lambda} X$ is not the same as $\gr_{(t+s)\lambda}X$ -- even in the case where $\lambda$ is a single curve. 
An intuitive reason for this is given by the following
observation: To obtain $\gr_{t\gamma}\gr_{s\gamma}X$ from $\gr_{s\gamma}X$, one has to replace the
geodesic representative $\gamma'$ of $\gamma$ on $\gr_{s\gamma}X$ with a flat cylinder of length
$t$ -- to obtain $\gr_{(t+s)\gamma}X$ from $\gr_{s\gamma}X$ on the other hand, one has to make the
already inserted grafting cylinder longer (by $t$); for example by cutting at the flat core 
curve $\delta = \gamma \times \{s/2\}$ of the already glued in grafting cylinder and then pasting
in another flat cylinder $\gamma \times [0,t]$.
However, a priori the curves $\gamma'$ and $\delta$ may be very different and thus the two
surgery operations will give different results.
Also note that even if $\gamma'$ and $\delta$ were identical curves, the surgery operations would
not yield the same result, as grafting is defined in terms of the cylinder length -- and since grafting 
decreases the length of the grafting curves, $\gamma'$ will be shorter than $\gamma$; hence the 
modulus of a length $t$ cylinder at $\gamma'$ will be larger than the modulus of a length $t$ cylinder
at $\gamma$.

The following two theorems are the main technical results of this paper.
\begin{theorem}[Iterating a multicurve]
  \label{quasiflow}
  Let $S$ be a closed surface of genus $g > 1$. There are constants $\widetilde{\epsilon}, C > 0$ such that the following holds: 

  Let $\lambda = t_1\gamma_1 + \ldots + t_n\gamma_n$ be a weighted multicurve on $S$ and $\eta=s_1\gamma_1+\ldots+s_n\gamma_n$ be another 
  multicurve with the same supporting curves.
Let $X \in \teich(S)$ be a hyperbolic structure such that the hyperbolic lenghts satisfy $l_X(\gamma_i) \leq \widetilde{\epsilon}$ for all $i$. 
Then
  $$d_\teich\left(\gr_{\eta}(\gr_{\lambda} X), \gr_{\eta\widetilde{+}\lambda}X\right) \leq C\cdot\left(\max_{i=1,\ldots,n}l_X(\gamma_i)\right)^{1/8}$$
  where $\eta\widetilde{+}\lambda$ is a ``weighted sum'' of $\lambda$ and $\eta$:
  $$\eta\widetilde{+}\lambda = \left(\frac{\pi+t_1}{\pi}\cdot s_1+t_1\right)\gamma_1 + \ldots + \left(\frac{\pi+t_n}{\pi}\cdot s_n+t_n\right)\gamma_n$$
\end{theorem}

\begin{theorem}[Splitting a multicurve]
  \label{quasicomm}
  Let $S$ be a closed surface of genus $g > 1$. There are constants $\widetilde{\epsilon}, C > 0$ such that the following holds: 

  Let $\lambda = t_1\gamma_1 + \ldots + t_n\gamma_n$ and $\eta = t_{n+1}\gamma_{n+1} + \ldots t_m\gamma_m$ be disjoint weighted multicurves on $S$. Let $X \in \teich(S)$ be a hyperbolic structure such that $l_X(\gamma_i) \leq \widetilde{\epsilon}$ for all $i=1,\ldots,m$. 

Then
  $$d_\teich\left(\gr_{\eta}(\gr_{\lambda} X), \gr_{\eta + \lambda}X\right) \leq C\cdot\left(\max_{i=1,\ldots,m}l_X(\gamma_i)\right)^{1/8}$$
\end{theorem}
Both theorems are proved by explicitly constructing a quasiconformal comparison map and estimating its dilatation.

After reviewing some basic facts from hyperbolic geometry (section 2) we define the building blocks for these
maps in section 3 and develop formulas to estimate their dilatation.
Section 4 is devoted to showing the main technical results (theorems 1.4 and 1.5 above). These proofs are divided into several steps which are outlined and explained in section 4.1.
In section 5 we then study holonomy lifts and obtain theorems 1.1 to 1.3.
As a last application, we study the asymptotic behaviour of grafting sequences $\gr^n_{\lambda}X$ in section 6.
We show that these sequence converge geometrically to a punctured surface for every base point $X$.

\vspace{2 mm} 
\textbf{ACKNOWLEDGEMENTS.} The author would like to thank his advisor Ursula Hamenstädt
for her considerable support throughout the project and David Dumas for interesting and helpful
discussions. Most of this work was done during a visit at the MSRI in Berkeley in fall 2007.
The author would like to thank the institute for its hospitality and the organizers of the 
semester programme on Teichmüller theory and Kleinian groups. He would also like to thank the Hausdorff Center for Mathematics in Bonn for its financial support that made the stay in Berkeley possible.

\section{Some hyperbolic geometry}
\label{sec:notation}
For convenience we recall some facts from elementary hyperbolic geometry which we will need in the sequel. In this paper we will always use the upper half plane model for the hyperbolic plane $\HH^2$.
\begin{figure}[htbp!]
  \centering
  \includegraphics[width=0.6\textwidth]{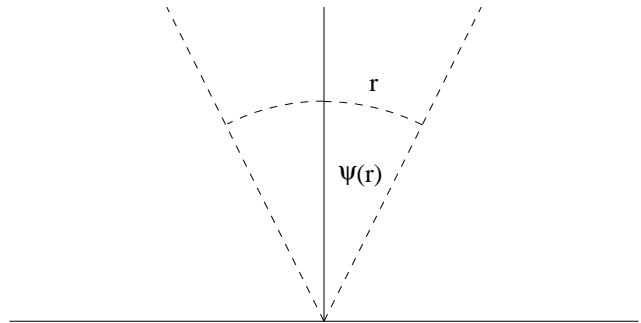}
  \caption{Regular neighbourhoods in $\HH^2$}
  \label{fig:rg}
\end{figure}

The hyperbolic regular $r$-neighbourhood of the imaginary axis $\{z\in\HH^2, d(z, i\RR)<r\}$  is an infinite circle sector bounded by two straight euclidean rays through the origin. We will call the angle between these rays and the imaginary axis the \textit{angle} $\psi(r)$ \textit{corresponding to the neighbourhood} (cf. figure \ref{fig:rg})

Similarly, if $A=\{z \in X, d(z, \gamma)<r\}$ is a embedded annulus around a simple closed geodesic $\gamma$ on a hyperbolic surface $X$, it can be lifted to a regular $r$-neighbourhood of the imaginary axis in $\HH^2$. 
We call the angle correponding to this lifted neighbourhood the \textit{angle corresponding to the annulus} $A$.

From elementary hyperbolic geometry we know
$$\psi(r) = \arctan\left(\frac{e^{2r}-1}{2e^r} \right).$$
We will often need a simple estimate for small $r$, namely
\begin{prop}[Estimate for annulus angles]
\label{psi_u}
  If $r$ is small enough, we have
  $$\psi(r) \leq r$$
\end{prop}

\begin{proof}
  As $\tan$ is increasing, it is enough to show that
  $$\frac{e^{2r}-1}{2e^r} \leq \tan(r)$$
  for small $r$. Taking derivatives we see that both sides agree up to order 2 at $r=0$. However, the third derivative at $r=0$ is 1 for the left hand side and 2 for the right hand side. Thus the inequality holds for small $r$.
\end{proof}

Recall that there is a function $M:\RR^+\to\RR^+$ such that if $\gamma$ is a simple closed geodesic of length $\leq l$ on any hyperbolic surface $X$, the regular $M(l)$-neighbourhood of $\gamma$ is an embedded annulus (this is the classical collar lemma, cf. \cite{Buser:oq})
We call this annulus the \textit{standard hyperbolic collar} and denote the corresponding angle by $\theta(l)$. A calculation yields
$$\theta(l) = \arccos\left( \frac{e^l-1}{e^l+1}\right).$$
In the sequel we will often have to estimate this quantity, in particular we need
\begin{prop}[Estimate for standard collars]
\label{theta_u}
  For small $l$ we have
  $$\frac{\pi - l}{2} \leq \theta(l).$$
\end{prop}
\begin{proof}
  As $\cos$ is decreasing for small arguments, it suffices to prove
  $$\cos\left(\frac{\pi-l}{2}\right) \geq \frac{e^l-1}{e^l+1}$$
  for small $l$.

  Taking derivatives we see that both sides agree up to order 2 at $l=0$. The third derivative of the left hand side is $-1/8$, while that of the right hand side is $-1/4$. Thus the right hand side decreases faster and the inequality holds for small $l$ as claimed.
\end{proof}

\section{Scaling, Shearing and Twisting maps}
\label{sec:sst}
The quasiconformal maps used in the proof of theorems \ref{quasiflow} and \ref{quasicomm} will be constucted out of simple building blocks, which we now describe.

A \textit{finite annulus} in the complex plane is a open domain $A$ bounded by two nonintersecting Jordan curves. The domain 
$\{z \in \CC, r<|z|<s\}$
bounded by two round circles is called a \textit{round annulus}. 
By the uniformization theorem, any finite annulus $A$ is biholomorphic to a round annulus. 
An uniformizing map extends to a homeomorphism of the closed annuli (cf. for example \cite[chapter I §2.2]{Lehto:1965fj}).
Thus, if $A$ and $B$ are finite annuli which are biholomorphic to round annuli $A'$ and $B'$, any $K$-quasiconformal mapping of the closures $\overline{A'} \to \overline{B'}$ gives rise to a quasiconformal mapping $\overline{A} \to \overline{B}$ (of the same dilatation) and vice versa.

Recall that the modulus $\Mod(A)$ of an annulus $A \subset \CC$ is the extremal length of the ``topological radii'' -- that is, the familiy of curves connecting the two boundary curves.
The modulus yields a complete classification of finite annuli: $A, B$ are biholomorphic if and only if they have the same modulus.
For a round annulus in the complex plane we have
$$\Mod\left(\{a<|z|<b\}\right) = \frac{1}{2\pi}\log\left(\frac{b}{a}\right).$$
Also recall the formula $\Mod(A) = \pi/l$, where $l$ is length of the simple closed geodesic with respect to the complete hyperbolic metric on $A$.

\subsection{Scaling}
Suppose $A, B \subset \CC$ are two round annuli. 
The problem of finding the optimal quasiconformal map $A\to B$ is classical. We want to describe its solution, which we call the \textit{scaling map} $s_{A,B}$.

To do so, it is useful to introduce logarithmic coordinates for round annuli. Consider the holomorphic map
$$f_a: (0,a)\times\RR \to \CC, \quad\quad (t,x) \mapsto e^{t+2\pi i \cdot x}$$
This map is a holomorphic universal covering map of the annulus $A = \{1<|z|<e^a\}$. A fundamental domain is of the form $(0,a)\times[0,1]$. We will call the induced coordinates on the annulus $A$ \textit{logarithmic coordinates for $A$}. Note that $a = \Mod(A)$ and that the $x$ is nothing but the argument of $f_a(t,x) \in \CC$. Also note that these coordinates extend to give coordinates of the closed annulus.

In logarithmic coordinates, the scaling map is given by
$$(0,b) \times \RR \to (0,a) \times \RR, \quad\quad (t,x) \mapsto \left(\frac{a}{b}t, x\right)$$
Clearly, this map has quasiconformality constant $\max(a,b)/\min(a,b)$ and is thus optimal (due to the geometric classification of quasiconformal maps).

\subsection{Shearing}
Now let $A$ be a round annulus in the complex plane. We want to construct a quasiconformal self-map of the closure $\overline{A}$ of $A$ realizing a given angular distortion on the outer (or inner) boundary circle, while fixing the other boundary. More precisely

\begin{prop}[Shearing maps]
\label{shearing}
  Suppose $A \subset \CC$ is a round annulus of modulus $a>1$. Let $f:[0,1]\to[0,1]$ be a $B$-bilipschitz, increasing continuously differentiable map with $B<2, f(0) = 0, f(1)=1$. 

Then there is a quasiconformal homeomorphism $S_f:\overline{A} \to \overline{A}$ satisfying
  \begin{enumerate}[i)]
  \item $S_f$ fixes the inner boundary: $S_f(0,x) = (0,x)$ (in logarithmic coordinates)
  \item $S_f$ realizes the distortion $f$ on the outer boundary: $S_f(a,x) = (a,f(x))$
  \item 
The quasiconformality constant of $S_f$ satisfies
    $$\log(K(S_f)) \leq C\cdot(B-1)$$
    for some universal constant $C$.
  \end{enumerate}
 The same result holds by symmetry if we reverse the roles of inner and outer boundary. 
\end{prop}
\begin{proof}
  We define the map on a fundamental domain $[0,a]\times[0,1]$ in logarithmic coordinates for (the closure of) $A$:
  $$S_f: (t,x) \mapsto \left(t, \left(1-\frac{t}{a}\right)x + \frac{t}{a}f(x)\right)$$
  and continue cyclically.

  First we note that $S_f$ actually is a homeomorphism: $S_f$ is differentiable with linearly independent partial derivatives (see below), so it is locally a homeomorphism; furthermore it is bijective, as 
$ \left(1-\frac{t}{a}\right)x + \frac{t}{a}f(x)$ is strictly increasing in $x$ for each fixed $t$ -- and thus bijective.

  To estimate the quasiconformality constants, we compute
  $$\frac{d}{dt}S_f = \left(1, \frac{f(x)-x}{a} \right)$$
  $$\frac{d}{dx}S_f = \left(0, \left(1-\frac{t}{a}\right) + \frac{t}{a}f'(x)\right)$$
  Therefore ($\partial, \overline{\partial}$ denote Wirtinger derivatives)
  $$\partial S_f = \frac{1}{2}\left(1+1-\frac{t}{a}+\frac{t}{a}f'(x), \frac{f(x)-x}{a} \right)
= \frac{1}{2}\left(2-\frac{1-f'(x)}{a}t, \frac{f(x)-x}{a} \right)$$
$$\overline{\partial}S_f = \frac{1}{2}\left(1-1+\frac{t}{a}-\frac{t}{a}f'(x), \frac{f(x)-x}{a} \right)
= \frac{1}{2}\left(\frac{1-f'(x)}{a}t, \frac{f(x)-x}{a} \right)$$
Now we need to estimate $\frac{|\overline{\partial}S_f|}{|\partial S_f|}$. To this end, note that as $f$ is $B$-bilipschitz and monotonically increasing, we have $B^{-1} \leq f'(x) \leq B$ for all $x$. Thus, writing $B=1+\epsilon$,
$$1-f'(x) \leq 1-B^{-1} = 1 - \frac{1}{1+\epsilon} = \frac{\epsilon}{1+\epsilon} \leq B-1$$
$$f'(x) - 1 \leq B-1$$
and thus, $|1-f'(x)| \leq B-1$ and 
$$|f(x)-x| \leq (B-1)|x|$$
Using this we obtain (recall that $a > 1$)
$$\frac{|\overline{\partial}S_f|^2}{|\partial S_f|^2} = 
\frac{ (1-f'(x))^2\frac{t^2}{a^2} + \frac{(f(x)-x)^2}{a^2} }{ (2-(1-f'(x))\frac{t}{a})^2 + \frac{(f(x)-x)^2}{a^2}  }
\leq \frac{ (B-1)^2 + (B-1)^2  }{ (2-(B-1))^2 }$$
This shows, that the map has an (analytic) quasiconformality constant of less than
$$k = \sqrt{2}\frac{B-1}{3-B}$$
From this, we obtain the geometric quasiconformality constant as $K=\frac{1+k}{1-k}$. Thus we have
$$K = \frac{ 3-B + \sqrt{2}(B-1)  }{ 3-B - \sqrt{2}(B-1) } = 1 + \frac{2\sqrt{2}B-2\sqrt{2} }{(3+\sqrt{2}) - (1+\sqrt{2})B}$$
Using $\log(1+y) \leq y$ this yields the claim.
\end{proof}

Conversely, we need a way to estimate the shearing introduced by univalent maps of annuli. Let
$$A_s = \{ z \in \CC, s<|z|<1 \}$$
denote a round annulus in the complex plane. 
\begin{lemma}[Controlling boundary distortion]
\label{distortion}
  Suppose $f: A_r \to A_s$ is a univalent holomorphic map preserving the outer boundary ($f(S^1) = S^1$). 

  Then $f|_{S^1}$ is $K$-Lipschitz with respect to the angular metric on $S^1$, where $K=\Mod(A_s)/\Mod(A_r)$.
\end{lemma}
\begin{proof}
  Using the Schwarz reflection principle we first extend $f$ to a holomorphic map
$$ F : A_r^+ \to  A_s^+$$
 where $A_r^+ = \{ z \in \CC, r<|z|<r^{-1} \}$.

  Now it suffices to show that $F'|_{S^1} \leq K$ -- indeed (by precomposiong with a rotation) we only need to show it for $F'(1)$. Furthermore we can assume that $F(1) = 1$ (by postcomposing with a rotation).

  The universal covering map for $A_r^+$, $\pi_r(z):\HH^2 \to A_r^+$ is given by
  $$\pi_r(z) = \exp\left(\log(-i\cdot z)\frac{2\pi}{l}i\right) = \exp\left(\log(-i\cdot z)\cdot2\Mod(A_r^+)\cdot i\right)$$
  $$\pi_r'(z) = \exp\left(\log(-i\cdot z)\cdot2\Mod(A_r^+)\cdot i\right)\cdot\left(-i\frac{1}{-iz}\cdot 2\Mod(A_r^+)\right)$$  
 where $\log$ is any branch of the natural logarithm on $\HH^2$ and $l$ the hyperbolic length of the core curve of $A_r^+$. 
 
  Lift $F$ to a map $\widetilde{F}:\HH^2 \to \HH^2$ of the universal covers fixing $i$: $\widetilde{F}(i) = i$.
  As the universal covering map is locally biholomorphic, we can compute the derivative of $F$ as
  $$F'(1) = (\pi_s)'(i)\widetilde{F}'(i)(\pi_r')^{-1}(1).$$
  However, by the usual Schwarz lemma, we have $|\widetilde{F}'(i)| \leq 1$, and thus
  $$|F'(1)| \leq |\pi_s'(i)|\cdot |\pi_r'(i)|^{-1}.$$
  But, 
  $$|\pi_r'(i)| = 2\Mod(A_r^+) = 4\Mod(A_r)$$
  and thus the lemma follows.
\end{proof}

\subsection{Twisting}
Again, let $A$ be some round annulus in the complex plane. We want to find a quasiconformal model for a twist on $A$.
\begin{prop}[Twist maps]
  \label{twistmaps}
  Suppose $A$ is a round annulus of modulus $a$ and let $k \in \RR$ (the amount of twisting) be given. Then there is a map $T_k:\overline{A} \to \overline{A}$ such that
  \begin{itemize}
  \item $T_k$ fixes the inner boundary: $T_k(0,x) = (0,x)$ (in logarithmic coordinates)
  \item $T_k$ realizes a twist by $k$: $T_k(a,x) = (a,x+k)$
  \item The quasiconformality constant of $T_k$ satisfies
    $$\log(K(T_k)) \leq \frac{2}{\sqrt{1+4\left(\frac{a}{k}\right)^2}-1}$$
  \end{itemize}
\end{prop}
\begin{proof}
  The twist map in logarithmic coordinates is given by
  $$T_k: (t, x) \mapsto \left(t, x+\frac{t}{a}k\right)$$
  To prove the proposition we now perform a computation similar to the one in the proof of proposition \ref{shearing}. In particular, we see
  $$\frac{d}{dt}T_k = \left(1, \frac{k}{a}\right), \quad\quad\frac{d}{dx}T_k = (0, 1)$$
  $$\partial T_k = \frac{1}{2}\left(1+1, \frac{k}{a} \right), \quad\quad
  \overline{\partial}T_k = \frac{1}{2}\left(1-1, \frac{k}{a}\right)$$
  $$|\overline{\partial}T_k(\partial T_k)^{-1}|^2 = \frac{\frac{k^2}{a^2}}{4 + \frac{k^2}{a^2}}
= \frac{1}{1+4\frac{a^2}{k^2}}$$
This gives
$$K = \frac{\sqrt{1+4\frac{a^2}{k^2}} + 1}{\sqrt{1+4\frac{a^2}{k^2}} - 1}
= 1+\frac{2}{\sqrt{1+4\frac{a^2}{k^2}}-1 }$$
which, again using $\log(1+y) \leq y$, yields the result.
\end{proof}

\section{The iteration and splitting theorems}
\label{sec:poqft}
In this section we prove the two main technical results concerning iterated grafting along a short multicurve.
\begin{theorem}[Iterating a multicurve]
  \label{quasiflow}
  Let $S$ be a closed surface of genus $g > 1$. There are constants $\widetilde{\epsilon}, C > 0$ such that the following holds: 

  Let $\lambda = t_1\gamma_1 + \ldots + t_n\gamma_n$ be a weighted multicurve on $S$ and $\eta=s_1\gamma_1+\ldots+s_n\gamma_n$ be another 
  multicurve with the same supporting curves. Let $X \in \teich(S)$ be a hyperbolic structure such that the hyperbolic length of the geoodesics $\gamma_i$ satisfies $l_X(\gamma_i) \leq \widetilde{\epsilon}$ for all $i$. Then
  $$d_\teich\left(\gr_{\eta}(\gr_{\lambda} X), \gr_{\eta\widetilde{+}\lambda}X\right) \leq C\cdot\left(\max_{i=1,\ldots,n}l_X(\gamma_i)\right)^{1/8}$$
  where $\eta\widetilde{+}\lambda$ is a ``weighted sum'' of $\lambda$ and $\eta$:
  $$\eta\widetilde{+}\lambda = \left(\frac{\pi+t_1}{\pi}\cdot s_1+t_1\right)\gamma_1 + \ldots + \left(\frac{\pi+t_n}{\pi}\cdot s_n+t_n\right)\gamma_n$$
\end{theorem}

\begin{theorem}[Splitting a multicurve]
  \label{quasicomm}
  Let $S$ be a closed surface of genus $g > 1$. There are constants $\widetilde{\epsilon}, C > 0$ such that the following holds: 

  Let $\lambda = t_1\gamma_1 + \ldots + t_n\gamma_n$ and $\eta = t_{n+1}\gamma_{n+1} + \ldots t_m\gamma_m$ be disjoint weighted multicurves on $S$. Let $X \in \teich(S)$ be a hyperbolic structure such that $l_X(\gamma_i) \leq \widetilde{\epsilon}$ for all $i=1,\ldots,m$. 

Then
  $$d_\teich\left(\gr_{\eta}(\gr_{\lambda} X), \gr_{\eta + \lambda}X\right) \leq C\cdot\left(\max_{i=1,\ldots,m}l_X(\gamma_i)\right)^{1/8}$$
\end{theorem}

\subsection{Notation and outline of the proof}
\label{sec:outline}

To prove the theorems, we will explicitly construct a \textit{comparison map} from $\gr_{\eta}(\gr_{\lambda} X)$ to $\gr_{\eta\widetilde{+}\lambda}X$ (from $\gr_\eta\gr_\lambda X$ to $\gr_{\eta+\lambda}X$ respectively) and estimate its dilatation.

Before beginning with a formal proof, we first outline the main ideas in the case of theorem \ref{quasiflow}
as well as introduce certain notation for curves which will be used throughout the proofs.
The construction of the comparison map is devided into several steps. We first look at the situation after grafting along $\lambda$ (cf. figure \ref{fig:step} for the case where $\lambda$ is a simple closed curve). 
\begin{figure}[htbp!]
  \centering
  \includegraphics[width=\textwidth]{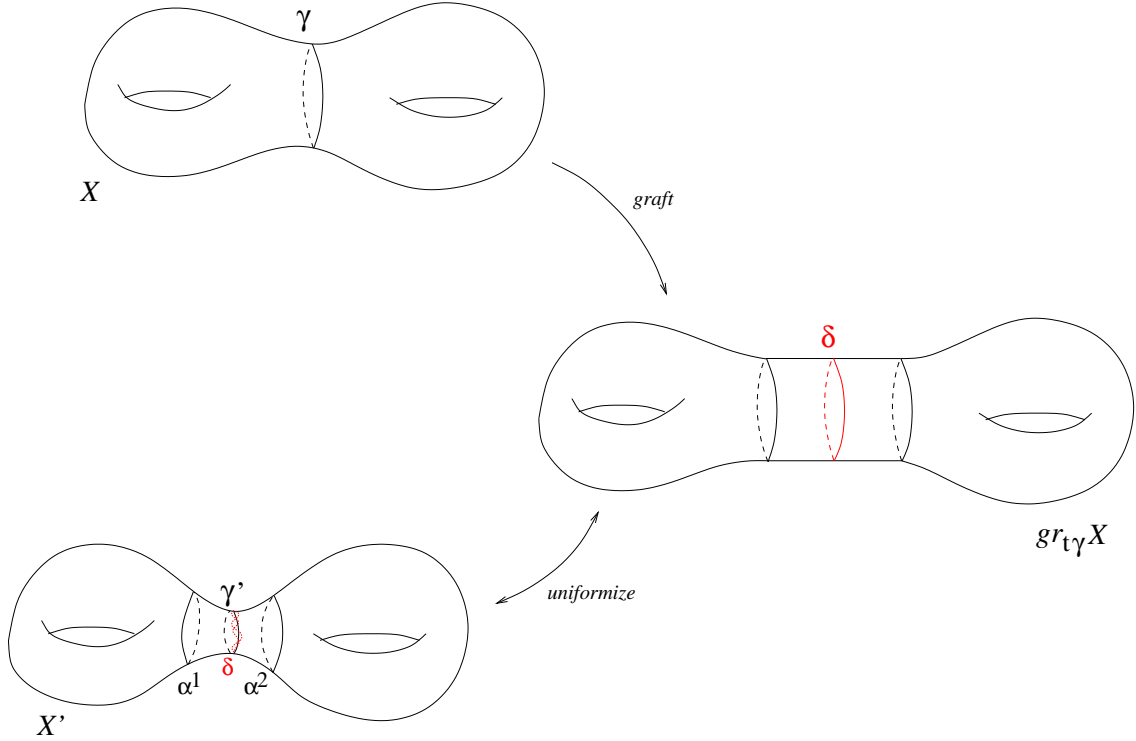}
  \caption{Grafting step}
  \label{fig:step}
\end{figure}
$X' = \gr_{\lambda}X$ is obtained from $X$ by cutting at the hyperbolic geodesics $\gamma_i$ in the
support of $\lambda$ and gluing in the flat cylinders $\gamma_i \times [0,t_i]$. We call these
cylinder the \textit{grafting cylinders} on $X'$ and the curve $\gamma_i \times \{t_i/2\}$ the
\textit{flat core curve} $\delta_i$ of the grafting cylinder corresponding to $\gamma_i$.
Denote by $\gamma'_i$ the geodesic representative of $\delta_i$ in the hyperbolic metric on $X'$.
To obtain $\gr_{\eta}(\gr_{\lambda} X)$ from $X'$, we have to cut $X'$ at the $\gamma'_i$ and
insert flat cylinders, whereas to obtain $\gr_{\eta\widetilde{+}\lambda}X$ from $X'$, we have to cut at the $\delta_i$.

Thus the first step is to see that $\gamma'_i$ and $\delta_i$ are close to each other in the
hyperbolic metric on $X'$. In section \ref{sec:bounding} we show that there is a
\textit{bounding annulus} of small, controlled modulus around $\gamma'_i$ which contains
$\delta_i$.
This is done by showing that the length of $\delta_i$ can be bounded from above, while the length of
the geodesic $\gamma'_i$ is bounded from below -- and thus, by elementary hyperbolic geometry, they
need to be close to each other.

Now we consider the situation at one grafting cylinder (and, for sake of simpler notation, we will
drop the index $i$ from the curves).
Let $\alpha^1$ and $\alpha^2$ be the boundary curves of the standard hyperbolic collar around
$\gamma'$ on $X'$ (cf. figure \ref{fig:step}). Once $\gamma$ is short, the bounding annulus
will be contained in this collar.

Thus we can construct \textit{pre-annulus maps} $\phi^j$, sending the annulus bounded by $\alpha^j$
and $\gamma'$ to the annulus bounded by $\alpha^j$ and $\delta$, which restrict to the identity on
$\alpha^j$. (section \ref{sec:amaps}).  By gluing these maps to the identity mapping on the
complement of the standard hyperbolic collar, we obtain a quasiconformal map
$$X' \setminus \gamma' \to X' \setminus \delta$$
with controlled dilatation (here, $X' \setminus \gamma$ denotes the Riemann surface with boundary
obtained by cutting at $\gamma$).  We then have to care about three issues.

First, the two pre-annulus maps $\phi^1$ and $\phi^2$ have to be modified to take the same values on
$\gamma'$ so that they can be glued to form a map of the surface $X'$ to itself.
Then, as we want to obtain a map from $\gr_{\eta}(\gr_{\lambda} X)$ to $\gr_{\eta\widetilde{+}\lambda}X$
without losing control over the quasiconformality constants, the pre-annulus maps have to be further
modified to send $\gamma'$ to $\delta$ in a way that is compatible with the respective grafting
operations (what this precisely means will be explained in detail in section \ref{sec:amaps}).
These two issues will be handled simultaneously by shearing $\phi^j$ by an appropriate amount,
obtaining \textit{annulus maps} $\Phi^j$ (section \ref{sec:amaps})

Finally, to estimate Teichmüller distance using this map, we have to make sure that it preserves the
marking on $X'$. The construction of the annulus maps may introduce a quite large unwanted twist --
which we compensate in a last step using an appropriate (un-)twist map (section \ref{sec:twist})

As all constructions took place just in the collar neighbourhood around $\gamma'$, we can repeat 
the arguments at all curves $\gamma_i$ to obtain a comparison map from $\gr_{\eta}(\gr_{\lambda} X)$ to $\gr_{\eta\widetilde{+}\lambda}X$. By tracing the error bounds of the involved maps we then conclude the theorem.
The construction for theorem \ref{quasicomm} is very similar; one procedes by showing that the
curves in $\eta$ neither change length nor position too much when grafting along $\lambda$ and then
constucting a comparison map as before.

\subsection{Lengths estimates and bounding annuli}
\label{sec:bounding}

We now construct the bounding annuli as sketched before. To do this we need to control the length of
the grafting curves after grafting along them once. For the proof it is convenient to show several statements
simultaneously
\begin{lemma}[bounding lemma]
  \label{deltabounding}
  Let $X$ be a hyperbolic surface and $\gamma_1, \ldots, \gamma_n$ be simple closed geodesics on
  $X$. Let $\lambda = t_1\gamma_1 + \ldots + t_n\gamma_n$ be a weighted multicurve.
  On the grafted surface $\gr_\lambda X$, consider the flat core curves $\delta_i$ of the grafting
  annuli and the hyperbolic geodesics $\gamma'_i$ in the free homotopy class of $\gamma_i$.
 
 Then there are constants $K_1, K_2, K_3 > 0$, depending only on the lengths of the $\gamma_i$ on
  $X$ such that the following statements hold.
\begin{description}
\item[i) length estimate] For all $i = 1, \ldots, n$ we have
$$K_1\cdot\frac{2\theta}{2\theta + t_i}\cdot l_X(\gamma_i) \leq l_{\gr_{\lambda}X}(\gamma'_i)\leq l_{\gr_{\lambda}X}(\delta_i) \leq \frac{\pi}{\pi + t_i}\cdot l_X(\gamma_i)$$
where $\theta$ is the angle corresponding to the standard collar neighbourhood around $\gamma_i$ on $X$.
If $\gamma_i$ is short enough, one can replace $K_1$ with $1/(1+l_X(\gamma_i))$.

\vspace{2 mm}
\item[ii) $\delta$-bounding annulus]
$\delta_i$ is contained in a hyperbolic $R_i$-tube around $\gamma'_i$ on $\gr_{\lambda}X$. Here, $R_i$ depends only on the length of
$\gamma_i$ and if $\gamma_i$ is short enough, we have $R_i \leq K_2 \cdot l_X(\gamma_i)^{1/4}$.

\vspace{2 mm}
\item[iii) seperation] Let $\gamma$ is a simple closed curve on $X$ disjoint from $\lambda$.  Denote
  by $\gamma^*$ the hyperbolic geodesic in the free homotopy class of $\gamma$ with respect to the
  hyperbolic metric on $\gr_{\lambda}X$.

Then the hyperbolic standard collar neighbourhood around $\gamma^*$ is disjoint from all grafting cylinders on $\gr_\lambda X$ and
$$K_1 \cdot l_X(\gamma) \leq l_{\gr_\lambda X}(\gamma^*) \leq l_X(\gamma)$$
If $\gamma$ is short enough, one can replace $K_1$ by $1/(1+l_X(\gamma))$.

\vspace{2 mm}
\item[iv) $\gamma$-bounding annulus] $\gamma$ is contained in a $R^*$-tube around $\gamma^*$ with
  respect to the hyperbolic metric of $\gr_\lambda X$, where $R^*$ depends only on the length of $\gamma$ and satisfies
 $R^* \leq K_3\cdot l_X(\gamma)^{1/4}$.
\end{description}
\end{lemma}
\begin{figure}[htbp!]
  \centering
  \includegraphics[width=0.5\textwidth]{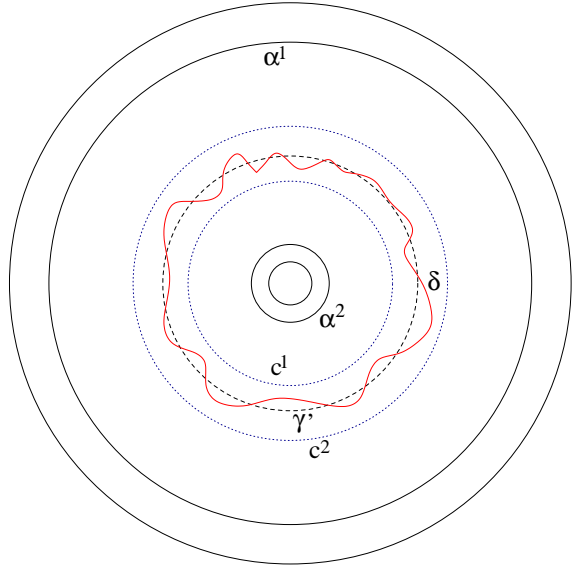}
  \caption{The situation in lemma \ref{deltabounding} \textit{ii)}. This figure depicts the (hyperbolic) annular cover of $\gr_{\lambda}X$ corresponding to $\gamma_i'$ (dotted central circle). $\alpha^1$ and $\alpha^2$ are the boundary curves for the standard collar around $\gamma'_i$.
$c^1$ and $c^2$ are the boundary curves of the $\delta$-bounding annulus}
\label{fig:collar}
\end{figure}

\begin{proof}
The lemma is proved by induction on the number of curves $n$ in the multicurve. One procedes as follows

\vspace{4 mm}
\textbf{i) for $n=1$}

This is a length estimate obtained by Diaz and Kim (Proposition 3.4 in \cite{KD07}). They show that
$$\frac{2\theta}{2\theta + \max(t_i)}\cdot l_X(\gamma_i) \leq l_{\gr_{\lambda}X}(\gamma'_i) \leq \frac{\pi}{\pi + t_i}\cdot l_X(\gamma_i)$$
which coincides with our (stronger) claim for $n=1$ and yields the upper bound on $ l_{\gr_{\lambda}X}(\gamma'_i)$ for all $n$. For convenience (and to emphasize why their length estimate also works for $\delta_i$) we will shortly summarize the proof given in \cite{KD07}. 

To see the upper bound one constructs an embedded holomorphic disc in the universal cover of $\gr_{\lambda}X$, such that the imaginary axis is sent to a lift of the curve $\delta_i$.
Identify the universal cover of $X$ with the hyperbolic plane $\HH^2$ and assume that the imaginary axis is a lift of $\gamma_i$.
To obtain the universal cover of $\gr_{\lambda}X$ from $\HH^2$, we have to cut along the imaginary axis and insert the sector $\{r\cdot e^{i\varphi}, r>0, \pi/2 \leq \varphi \leq \pi/2 + t_i\}$ (the resulting surface is to be understood multi-sheeted for large values of $t_i$) and then repeat the same picture equivariantly at other lift of all of the $\gamma_j$.
In particular, the map $z \mapsto z^{\frac{\pi+t_i}{\pi}}$ yields an embedding of $\HH^2$ in the universal cover of $\gr_{\lambda}X$. 
The image of the straight arc connecting $i$ and $e^{\frac{\pi}{\pi + t_i}l_X(\gamma_i)}i$ under this map projects to $\delta_i$ on $\gr_{\lambda}X$.
 As holomorphic maps are contracting with respect to the hyperbolic metrics, this gives the estimate
$$l(\delta_i) \leq \frac{\pi}{\pi+t_i}l_X(\gamma_i)$$
which yields the upper bound.

The lower bound follows by explicitly constructing a quasiconformal mapping between $\gr_{\lambda} X$ and $X$ (by collapsing the grafting cylinders) to estimate the Teichmüller distance and then using a lemma by Wolpert (\cite[lemma 3.1]{Wolpert:1979fj}) to relate Teichmüller distance and hyperbolic length ratios and obtain the claim.

\vspace{4 mm}
\textbf{i) for some n $\Rightarrow$ ii) for n}

The fact that $\delta_i$ and $\gamma'_i$ are freely homotopic allows us to estimate the distance from any point $\delta_i(s)$ on $\delta_i$ to $\gamma_i'$ in terms of their lengths (cf. \cite[Chapter 2, Theorem 2.23]{McMullen:1994uo})
$$\cosh^2d(\delta_i(s),\gamma'_i) \leq \frac{\cosh^2(l(\delta_i)/2)-1}{\cosh^2(l(\gamma'_i)/2) - 1}$$
In particular, with the upper bounds for $l(\delta_i)$ and the lower bounds for $l(\gamma'_i)$ from \textit{i)}, we can use this formula to obtain the constants $R_i$ which define the neighbourhoods with the claimed property.

It remains to show the estimate for $R_i$. To this end, recall the estimate for $\theta$ obtained in Proposition \ref{theta_u} to find using \textit{i)}
$$l(\gamma'_i) \geq K_1 \cdot \frac{2\theta}{2\theta + t_i}l_i \geq K_1\cdot\frac{\pi-l_i}{\pi-l_i+t_i}$$
where $l_i = l_X(\gamma_i)$. For the rest of the computation, we will drop the index $i$ of $l, t, \gamma$ and $\delta$ to make the formulae easier to read.
We claim that there is a $K$ such that
$$ \frac{\cosh^2(l(\delta)/2)-1}{\cosh^2(l(\gamma')/2) - 1} \leq \left(1 + K\sqrt{l}\right)^2$$
Note that comparing derivatives at $x=0$ yields
$$1+x^2 \leq \cosh^2(x) \leq 1+x^2+\frac{1}{2}x^4$$
near $x=0$ and thus
$$ \frac{\cosh^2(l(\delta)/2)-1}{\cosh^2(l(\gamma')/2) - 1} \leq \frac{ (l(\delta)/2)^2 +\frac{1}{2}(l(\delta)/2)^4 }{(l(\gamma')/2)^2}.$$
Using the estimate above and $l(\delta) \leq \pi/(\pi+t)\cdot l$ we obtain
$$\leq \frac{ \left(\frac{\pi}{\pi+t}\right)^2\left(\frac{l}{2}\right)^2 + \frac{1}{2}\left(\frac{\pi}{\pi+t}\right)^4\left(\frac{l}{2}\right)^4}{ K_1^2\left(\frac{\pi-l}{\pi-l+t}\right)^2\left(\frac{l}{2}\right)^2}
= \frac{1}{K_1^2}\cdot\frac{\left(\frac{\pi}{\pi+t}\right)^2 + \frac{1}{2}\left(\frac{\pi}{\pi+t}\right)^4\left(\frac{l}{2}\right)^2 }{\left(\frac{\pi-l}{\pi-l+t}\right)^2}$$
Next, we note that
$$\left(\frac{\pi}{\pi+s} \right)\left(\frac{\pi+s-l}{\pi-l}\right) \leq \frac{\pi}{\pi-l} = 1 + \frac{l}{\pi - l}$$
and that $1/K_1^2 \leq (1+l)^2$ for short curves $\gamma$ (compare \textit{i)}). This allows to further estimate 
$$ \frac{\cosh^2(l(\delta)/2)-1}{\cosh^2(l(\gamma')/2) - 1} \leq 
(1+l)^2\left( \left(1+\frac{l}{\pi-l}\right)^2 + \frac{1}{2}\left(\frac{\pi}{\pi+t}\right)^2 \left(1+\frac{l}{\pi-l}\right)^2\left(\frac{l}{2}\right)^2 \right) $$
As $\gamma$ is assumed to be short, $l$ is bounded, and thus the right hand side is smaller than $1+K'\cdot l$ for some $K'$ (independent of $t$, as $\pi/(\pi+t) \leq 1$). But then there is a $K$ such that $1+ K'\cdot l \leq (1+K\sqrt{l})^2$ which proves our claim. 
From what we have seen up to now we know that 
$$R_i \leq \arccosh\left(1+K\sqrt{l}\right)$$
But $\cosh(x) \geq 1+ \frac{1}{2}x^2$, and thus
$$ 1 + K\sqrt{l} \leq \cosh\left(K_2\cdot l^{1/4}\right)$$
for an appropriate $K_2$ -- which proves statement \textit{ii)}.

\vspace{4 mm}
\textbf{i), ii) for some n $\Rightarrow$ iii) for n}

The idea is to show that the grafting cylinders around $\delta_i$ are contained in the hyperbolic collar neighbourhoods of the $\gamma_i'$. By the usual collar lemma this shows that the collar around $\gamma^*$ is disjoint from the grafting cylinders. 
To do so, we estimate the hyperbolic width of the grafting cylinders and compare this to the width of the standard collars.

Let $C_i$ be the \textit{extended grafting cylinder} around $\gamma_i'$ -- the union of the standard hyperbolic collar neighbourhood of $\gamma_i$ with the grafting cylinder. Its modulus is given by
$$\Mod(C_i) = \frac{2\theta + t_i}{l_X(\gamma_i)}$$
This can be seen by considering the universal covering of $X$ by $\HH^2$ such that the imaginary axis is a lift of $\gamma_i$. Then the collar neighbourhood lifts to a regular neighbourhood of this axis (which is a infinite circle segment with vertex angle $2\theta$). Grafting at $\gamma_i$ by $t_i$ amounts to inserting ``lunes'' at each lift of the $\gamma_i$ (see section 2 of \cite{McMullen:1998mz}). At the imaginary axis, such a lune is just an euclidean circle sector with vertex angle $t_i$, and therefore the extended grafting cylinder is obtained as the projection of a circle sector of vertex angle $2\theta + t_i$. As the modulus of such a cylinder is given by $\log(w)/\phi$ where $w$ is the euclidean width and $\phi$ the total angle, the claim follows.
Also note that the core geodesic with respect to the complete hyperbolic metric of $C_i$ is just $\delta_i$ and the grafting cylinder is a round subannulus of $C_i$. 
In the following, we will again drop the index $i$ and denote $l_X(\gamma_i)$ by $l$.

For the complete hyperbolic metric of the extended grafting cylinder, the length of the core curve is given by 
$$\widetilde{l} := l_C(\delta) = \pi/\Mod(C) = \frac{\pi}{2\theta + t}l$$
Consider now the universal covering $\HH^2 \to C$ such that the imaginary axis is a lift of the core curve $\delta$. The grafting cylinder in $C$ then lifts to a regular neighbourhood of this axis (as it is a round subannulus). If $2\varphi$ denotes the vertex angle of this segment, the modulus of the grafting cylinder is given by $2\varphi / \widetilde{l}$. As we know the modulus of the grafting cylinder ($t_i/l$) we obtain 
$$\varphi = \frac{\pi}{2}\cdot\frac{t}{t+2\theta}, \quad\quad \phi := \pi/2 - \varphi =  \frac{\pi}{2}\cdot\frac{2\theta}{2\theta + t}$$
Therefore the hyperbolic distance of $\delta$ to the boundary curves (in the complete metric of the extended grafting cylinder) is given by
$$B := \log\frac{\cos(\phi/2)}{\sin(\phi/2)}$$
However, as holomorphic maps between complete hyperbolic surfaces are contracting, $B$ also bounds the distance of $\delta$ to the boundary of the grafting cylinder in the hyperbolic metric of $\gr_\lambda X$.

On the other hand, the width of the standard collar is given by (see e.g. section 3.8 of \cite{Hubbard:2006fk})
$$M(x) = \frac{1}{2}\log\frac{\cosh(x/2)+1}{\cosh(x/2)-1}$$
By \textit{i)} we know that $l_{\gr_\lambda X}(\gamma_i') \leq l' := \frac{\pi}{\pi+t}l$ and thus the collar width is larger than $M(l')$.
By \textit{ii)} we know that $\delta \subset {\mathcal U}_R(\gamma')$ where $R = K_2\cdot l^{1/4}$ and therefore the grafting cylinder is contained in a $(B+R)$-neighbourhood of $\gamma'$. Thus we are done once we show that $R+B \leq M(l')$.
Explicitly, this means
$$\log\left(\frac{\cos\phi/2}{\sin\phi/2}\right) + K_2\cdot l^{1/4} \leq \frac{1}{2}\log\frac{\cosh(l'/2)+1}{\cosh(l'/2)-1}$$
$$\Leftrightarrow \quad\quad e^{2K_2\cdot l^{1/4}} \frac{1}{(\tan(\phi/2))^2} \leq \frac{\cosh(l'/2)+1}{\cosh(l'/2)-1}$$
Note that $(\tan^2)'' > 0$ globally, and thus $\tan(x)^2 \geq x^2$. On the other hand, defining 
$$h(x) = \frac{\cosh(x/2)-1}{\cosh(x/2)+1}$$
we find $h(0)=0$, $h'(0) = 0$, $h''(0) = 1/8$, $h'''(0)=0$ and $h^{(iv)}(0) < 0$ and therefore, for small $x$, we have $h(x) \leq \frac{1}{16}x^2$. Hence it suffices to show
$$e^{2K_2\cdot l^{1/4}}\frac{4}{\phi^2} \leq 16 \cdot l'^{-2} \quad\Leftrightarrow\quad
e^{2K_2\cdot l^{1/4}}l'^2 \leq 4\phi^2$$
Recalling the definitions of $l'$ and $\phi$, this follows once
$$e^{2K_2\cdot l^{1/4}}\left(\frac{\pi}{\pi+t}\right)^2l^2 \leq 4\frac{\pi^2}{4}\left(\frac{2\theta}{2\theta+t}\right)^2$$
As $\theta < \pi/2$ this is true, once $e^{2K_2\cdot l^{1/4}}l^2 \leq (2\theta)^2$ -- which will be satisfied, once $l$ is small enough.

It remains to show the length estimate. To this end, denote the hyperbolic standard collar around $\gamma^*$ on $\gr_\lambda X$ by $A$. The modulus of $A$ is determined by the length of $\gamma^*$ alone, more precisely we have
$$\Mod(A) = \pi\left( 1 - \frac{4}{\pi}\arctan\left(\frac{e^{l/2}-1}{e^{l/2}+1}\right) \right)\cdot\frac{1}{l}$$
where $l = l_{\gr_\lambda X}(\gamma^*)$, as can be checked using the explicit form of $M$ given above.

As $A$ is disjoint from all grafting cylinders, we see the same annulus on $X$: there is a holomorphic inclusion $A \to X$ (by removing the grafting cylinders). As holomorphic maps are contracting, we have
$$l_X(\gamma) \leq l_A(\gamma^*) = \frac{\pi}{\Mod(A)} \leq \frac{l_{\gr_\lambda X}(\gamma^*)}{K}$$
with $K = K(l) = 1 - \frac{4}{\pi}\arctan\left(\frac{e^{l/2}-1}{e^{l/2}+1}\right)$; so $l_{\gr_\lambda X}(\gamma^*) \geq K\cdot l_X(\gamma)$.

The equation $l_X(c) \geq l_{\gr_\lambda X}(c)$ is true for any curve $c$ which does not intersect $\lambda$ (see e.g. theorem 3.1. in \cite{McMullen:1998mz}). 

Finally, we need to prove that $K \geq \frac{1}{1+l}$ for short $\gamma$. But this follows, as $K(0) = 1$ and $K'(0) = -1/\pi$, which shows the claim for small $l$.

\vspace{4 mm}
\textbf{i), iii) for some n $\Rightarrow$ i) for (n+1)}

As explained in the step \textit{i) for $n=1$}, the right hand side inequality in \textit{i)} is already known for all $n$.. It remains to show the left hand side.
To this end, write $\lambda = \lambda' + s\cdot \gamma$ where $\lambda'$ is a multicurve with $n$ curves and $\gamma$ is disjoint from $\lambda'$.

By collapsing the extended grafting cylinder to the standard collar, one sees
$$d_\teich\left(\gr_{\lambda}X, \gr_{\lambda'}X\right) \leq \frac{1}{2}\log\left(\frac{2\theta}{2\theta + s}\right)$$
By a lemma of Wolpert \cite[lemma 3.1]{Wolpert:1979fj} this implies for the lengths
$$l_{\gr_\lambda X}(\gamma) \geq \frac{2\theta}{2\theta + s}l_{\gr_\lambda' X}(\gamma)$$
But, as $\gamma$ is disjoint from $\lambda'$ the length estimate from \textit{iii)} yields
$$l_{\gr_\lambda X}(\gamma) \geq K_1\frac{2\theta}{2\theta + s}l_{X}(\gamma)$$
By applying the argument to all $\gamma$ in $\lambda$ the claim follows.

\vspace{4 mm}
\textbf{i), iii) for some n $\Rightarrow$ iv) for n}

This claim is proven analogous to the step \textit{i) for n} $\Rightarrow$ \textit{ii) for n}. We know that $l_{\gr_\lambda X}(\gamma) \leq l_X(\gamma) =: l$ and $l_{\gr_\lambda X}(\gamma^*) \geq K_1\cdot l_X(\gamma)$. One now computes
$$\frac{\cosh^2(l(\gamma)/2)-1}{\cosh^2(l(\gamma^*)/2) - 1} \leq \frac{ (l(\gamma)/2)^2 +\frac{1}{2}(l(\gamma)/2)^4 }{(l(\gamma^*)/2)^2}$$
$$\leq \frac{\left(\frac{l}{2}\right)^2+ \frac{1}{2}\left(\frac{l}{2}\right)^4}{K^2\left(\frac{l}{2}\right)^2} \leq (1+l)^2\left(1+\frac{1}{2}\left(\frac{l}{2}\right)^2 \right) \leq 1 + C\cdot l.$$
From here, one concludes the proof exactly as in the step \textit{i) $\Rightarrow$ ii)}.
\end{proof}

Statements \textit{ii)} (respectively \textit{iv)}) of the preceding lemma show that from the point
of view of the hyperbolic metric on the grafted surface, the $\gamma'_i$ and $\delta_i$
(respectively $\gamma$ and $\gamma^*$) are not far apart.

We will need another formulation of this fact in terms of moduli. To this end, note that the $R_i$
(respectively $R^*$) neighbourhood will be an embedded annulus once $R_i$ ($R^*$) is small
enough. As the collar width increases for shorter curves, while $R_i$ and $R^*$ decrease, there is a
constant $\epsilon>0$ such that for $\epsilon$-short curves the $R_i$ ($R^*$) neighbourhoods will be
embedded annuli in the hyperbolic collars.
For the construction of the comparison maps we will always assume that this condition is satisfied.
In the sequel, we will call these neighbourhoods \textit{bounding annuli}.

The boundary $\{z \in \gr_{\lambda}X, d(z,\gamma'_i) = R_i\}$ of such an annulus then consists of two curves, which in the sequel we will denote by $c_i^1$ and $c_i^2$ (or, $c_*^1, c_*^2$ in the $\gamma^*$-bounding case).
Furthermore, denote the boundary curves of the hyperbolic collar around $\gamma_i'$ (resp. $\gamma^*$) by $\alpha_i^1$ and $\alpha_i^2$ in such a way, that $c_i^1$ ($c_*^1$) is the component which is closer to $\alpha_i^2$ in the standard hyperbolic collar (also cf. figure \ref{fig:collar})

Using this notation, we obtain
\begin{lemma}[Modulus of bounding annuli]
\label{deltamod}
\begin{enumerate}[i)]
\item  Let $C^j_i$ be the annulus bounded by $\alpha^1_i$ and $c^j$.

The moduli of $C_i^j$ satisfy (for small $l_X(\gamma_i)$)
    $$\Mod(C_i^1) \leq \frac{\theta(l'_i) + R_i}{l'_i}, \quad\quad \Mod(C_i^2) \geq \frac{\theta(l'_i)-R_i}{l'_i}$$
    where $l'_i = l_{\gr_{s\lambda}}(\gamma_i)$ and $l_i = l_X(\gamma_i)$ and $R_i$ is the constant from lemma \ref{deltabounding} \textit{ii)}.
  In particular
    $$\frac{\Mod(C_i^1)}{\Mod(C_i^2)} \leq \frac{\theta(l_i)+R_i}{\theta(l_i)-R_i}$$

 By symmetry, the same result holds if we replace $\alpha_i^1$ by $\alpha_i^2$ and switch the roles of $C_i^1$ and $C_i^2$.
\item Let $C_*^j$ be the annulus bounded by $\alpha_*^1$ and $c^j$.
Then the inequalities from \textit{i)} hold, replacing $R_i$ by $R^*$, where $R^*$ is the constant from \ref{deltabounding} \textit{iv)}.
\end{enumerate}
\end{lemma}
\begin{proof}
   Recall that $\psi(r)$ denotes the angle corresponding to an diameter-$2r$-annulus (cf. section \ref{sec:notation}). Then by definition of $c^1_i$ and $c^2_i$ and \ref{deltabounding} \textit{ii)} we have (as all involved annuli are round subannuli of the annular cover, and thus moduli behave additively)
   $$\Mod(C^1_i) = \frac{\theta(l'_i) + \psi(R_i)}{l'_i}, \quad\quad \Mod(C_i^2) = \frac{\theta(l'_i) - \psi(R_i)}{l'_i}$$
   As $\psi(r) \leq r$ for small $r$ (cf. proposition \ref{psi_u}) we see the first claim.
   Then we also have
$$\frac{\Mod(C_i^1)}{\Mod(C_i^2)} \leq \frac{\theta(l(\gamma'_i)) + R_i}{\theta(l(\gamma'_i)) - R_i}$$
   But, as the map $x \mapsto (x+1)/(x-1)$ is decreasing, and $\theta(l(\gamma')) \geq \theta(l)$, we conclude \textit{i)}.
The statement in \textit{ii)} follow in the exact same way, using \ref{deltabounding} \textit{iv)} instead of \textit{ii)}.
\end{proof}

\subsection{(Pre-)Annulus maps}
\label{sec:amaps}
Now we are set to construct the pre-annulus and annulus maps as sketched in section \ref{sec:outline}.
We handle the cases for theorem \ref{quasiflow} and \ref{quasicomm} simultaneously. 
In order to do so, let us fix some notation for the rest of this section.

Suppose $X$ is a hyperbolic surface, $\lambda = t_1\gamma_1 + \ldots + t_n\gamma_n$ a weighted
multicurve. Let $X' = \gr_{\lambda}X$.
Let $\gamma$ be either one of the $\gamma_i$ or a simple closed geodesic disjoint from
$\lambda$. Denote by $\delta$ either the flat core curve of the grafting cylinder corresponding
to $\gamma=\gamma_i$ or the (image of the) curve $\gamma$ on $X'$ respectively.

Let $\gamma'$ be the hyperbolic geodesic on $X'$ in the free homotopy class of $\gamma$. Denote
the boundary curves of the standard hyperbolic collar around $\gamma'$ by $\alpha^1$ and
$\alpha^2$. (also cf. figure \ref{fig:collar} and \ref{fig:amaps})
Suppose the length of $\gamma$ is short enough to ensure that the bounding annulus on $X'$ is
contained in this standard hyperbolic collar around $\gamma'$.

Using this notation, we have
\begin{lemma}[pre-annulus maps]
\label{preannulus}
There are maps $\phi^j$ ($j=1,2$) with the following properties
   \begin{enumerate}[i)]
   \item $\phi^j$ is defined on the (closed) annulus bounded by $\alpha^j$ and $\gamma'$
   \item $\phi^j$ is an homeomorphism onto the annulus bounded by $\alpha^j$ and $\delta$.
   \item $\phi^j$ restricted to $\alpha^j$ is the identity.
   \item $\phi^j$ is quasiconformal, with quasiconformality constant $K$ satisfying
     $$log(K) \leq C\cdot l_X(\gamma)^{1/4}$$
     for some universal constant $C$.
   \end{enumerate}
\end{lemma}

\begin{proof}
  To describe the construction of the map, we look at the hyperbolic annular cover of $X'$ corresponding to $\gamma'$ (as depicted in figure \ref{fig:collar}).
The annular cover is the unique holomorphic covering map $p:X'_{\gamma'} \to X'$ with $\pi_1(X'_{\gamma'}) = \ZZ$ and $\im(p_*(\pi_1(X'_{\gamma'}))) = \left<[\gamma']\right> \subset \pi_1(X')$.

As $X'_{\gamma'}$ is an annulus, it can be (biholomorpically) embedded into
$\CC$ as a round annulus. This embedding yields coordinates for the hyperbolic collar
neighbourhood which are well-suited for our construction. In these
coordinates both $\alpha^1, \alpha^2$ and $\gamma'$ correspond to round
circles, which (by slight abuse of notation) we will also denote by the same
symbol.

  Normalize, such that $\alpha^1$ becomes the unit circle $S^1 \subset \CC$. Then $\gamma' = r\cdot S^1$ for some $r<1$.  

  Denote by $f$ the biholomorphic map sending the annulus bounded by $S^1$ and $\delta$ to a round annulus with $S^1$ as outer boundary circle and $s\cdot S^1$ as inner. By Schwarz reflection, $f$ extends to a map of the boundary curves (which are analytic) and without loss of generality, we can assume that $f(S^1)=S^1, f(1)=1$.
  \begin{figure}[htbp!]
    \centering
    \includegraphics[width=\textwidth]{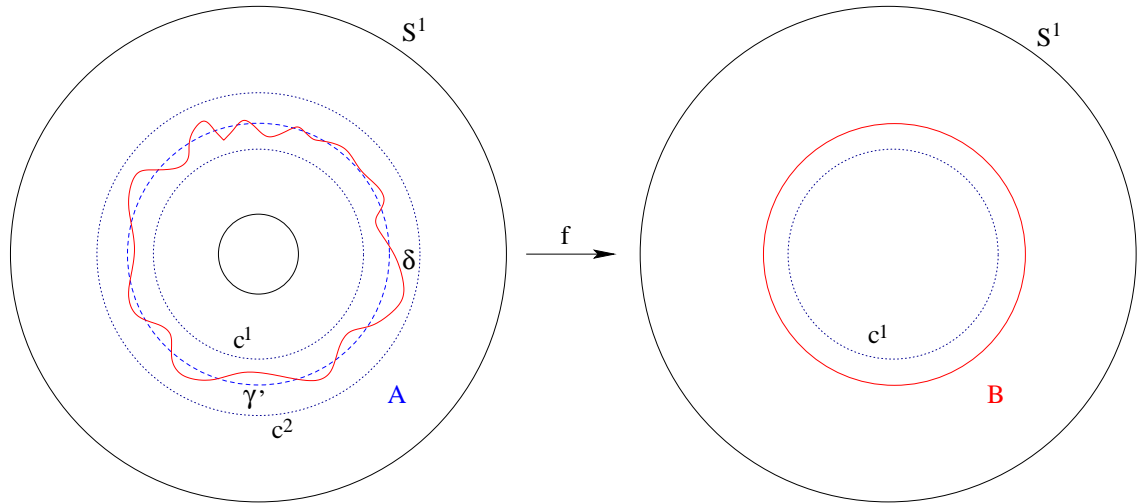}
    \caption{Constructing the annulus map}
    \label{fig:amaps} 
  \end{figure}

Denote the annulus bounded by the unit circle and $\gamma' = r\cdot S^1$ (on the left hand side of figure \ref{fig:amaps}) by $A$ and the one bounded by the unit circle and $f(\delta)=s\cdot S^1$ (on the right hand side) by $B$.

The annulus map will be defined as a composition of a scaling and a shearing (cf. section \ref{sec:sst}), namely
$$\phi^1 = f^{-1}\circ S_{f}\circ s_{A,B}$$
Here (by abuse of notation) $f$ also denotes the lift of $\left.f\right|_{S^1}$ to logarithmic coordinates, and the shearing map is taken with respect to the outer boundary. Note that $f$ is increasing (in logarithmic coordinates), as it is orientation preserving (as a biholomorphic map $B \to f(B)$). By our normalization it also fixes 0.

By construction, this map will satisfy conditions \textit{i)} to \textit{iii)}. By the propositions from section \ref{sec:sst} the quasiconformality constants depend on the quotient of the moduli of $A$ and $B$ (for the scaling part) and the Bilipschitz constant of $f$ in angular coordinates (for the shearing).

Now note that
$$\frac{\Mod(C^2)}{\Mod(C^1)} \leq \frac{\Mod(B)}{\Mod(A)} \leq \frac{\Mod(C^1)}{\Mod(C^2)}$$
where $C^i$ are the annuli bounded by the unit circle and $c^i$ (cf. figure \ref{fig:amaps})

On the other hand, $f$ maps $C^2$ into $C^1$, so by lemma \ref{distortion} (Controlling boundary distortion), $f$ is Lipschitz on $S^1$ with maximal dilatation $\Mod(C^1)/\Mod(C^2)$. $f^{-1}$ maps $B$ into $C^1$, so the same lemma yields that $f^{-1}$ is Lipschitz with dilatation $\Mod(C^1)/\Mod(B) \leq \Mod(C^1)/\Mod(C^2)$. So in fact, $f$ is Bilipschitz with this constant.

The scaling dilation satisfies
$$\log(K(S_{A,B})) \leq \log\left(\frac{\Mod(C^1)}{\Mod(C^2)}\right) \leq B-1$$
for $B = \Mod(C^2)/\Mod(C^1)$. The constant for the shearing satisfies (proposition \ref{shearing} (shearing maps))
$$\log(K(S_f)) \leq (\mathrm{const.})(B-1)$$
Using the previous estimates (lemmas \ref{deltamod} (Modulus of bounding annuli) and \ref{deltabounding} (bounding lemma)) we find for the appropriate constant $R$
$$B = \frac{\Mod(C^1)}{\Mod(C^2)} \leq \frac{\theta(l)+R}{\theta(l)-R}
= 1+\frac{2R}{\theta(l)-R} \leq 1+(\mathrm{const.})l^{1/4}$$
which yields the desired result.
\end{proof}
The two pre-annulus maps $\phi^1, \phi^2$ constucted in the proof of the previous lemma glue with the identity on the
complement of the standard collar to a map
$$(\gr_{\lambda} X) \setminus \gamma' \to (\gr_{\lambda} X)\setminus \delta$$
of the same quasiconformality constant (here, by $Z \setminus \gamma$ we denote the Riemann surface
with boundary obtained by cutting $Z$ at $\gamma$)
In order to extend these maps to the corresponding grafting rays, we have to modify them to
have a compatible behaviour in sending $\gamma'$ to $\delta$.

Consider first the situation of theorem \ref{quasiflow}. To obtain $\gr_{\eta}\gr_{\lambda}X$
from $\gr_{\lambda}$, we have to glue in flat cylinders of height $s_i$ and circumference
$l(\gamma_i')$ at the hyperbolic geodesics $\gamma'_i$ (with matching length parameters).
To obtain $\gr_{\eta\widetilde{+}\lambda}X$ from $\gr_{\lambda}$, we glue in euclidean cylinders at the
flat core curves $\delta_i$ of the already inserted grafting cylinders (again, with matching length
parameters \textit{in the flat metric of the grafting cylinders})

Thus, in the case $\gamma = \gamma_i$, we want to modify the pre-annulus map $\phi^j$ such that it
sends the curve $\gamma'$ parametrized by $S^1$ in constant speed \textit{in hyperbolic coordinates}
to the curve $\delta$ parametrized by $S^1$ in constant speed \textit{in the flat metric of the
  already inserted grafting cylinder}.

Similarly, in the situation of theorem \ref{quasicomm}, we obtain $\gr_\eta\gr_\lambda X$ from
$\gr_\lambda X$ by inserting flat cylinders at the geodesic representatives of $\eta$ on
$\gr_\lambda X$, whereas to obtain $\gr_{\eta+\lambda}X$ we need to insert them at the ``old''
geodesic representatives of $\eta$ given by the hyperbolic metric on $X$. Hence, in this case, we
want to modify the pre-annulus map to send $\gamma'$ parametrized in constant hyperbolic speed
\textit{with respect to $\gr_\lambda X$} to $\gamma$ parametrized in constant hyperbolic speed
\textit{with respect to $X$}.

In both cases, we call parametrizations of $\gamma, \gamma'$ and $\delta$ by $S^1$ with constant
speed in the respective metrics the \textit{natural parametrizations}.
Using this terminology, we have
\begin{lemma}[annulus maps]
 \label{annulusmaps}
 There are maps $\Phi^j$ ($j=1,2$) with the following properties:
  \begin{enumerate}[i)]
  \item $\Phi^j$ is defined on the (closed) annulus bounded by $\alpha^j$ and $\gamma'$
  \item $\Phi^j$ is a homeomorphism onto the annulus bounded by $\alpha^j$ and $\delta$.
  \item $\Phi^j$ restricted to $\alpha^j$ is the identity.
  \item $\Phi^j$ maps $\gamma'$ in its
    natural parametrization to $\delta$ in its natural parametrization
  \item $\Phi^j$ is quasiconformal, with quasiconformality constant $K$ satisfying
    $$log(K) \leq C\cdot l(\gamma)^{1/4}$$
    for some universal constant $C$.
  \end{enumerate}
\end{lemma}

\begin{proof}
\begin{figure}[htpb!]
  \centering
  \includegraphics[width=\textwidth]{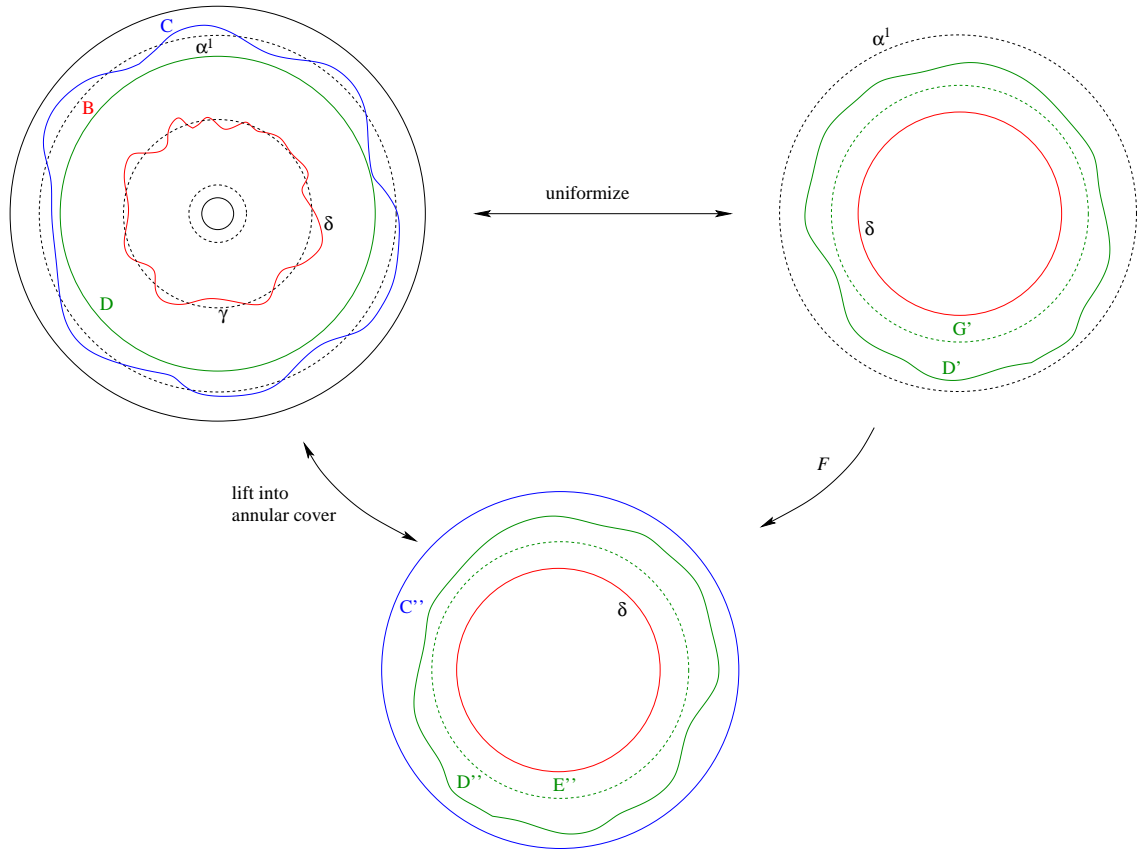}
  \caption{Comparison of the hyperbolic annular cover (upper left corner) to the 
natural charts around $\delta$ (on the bottom). All labelled annuli have $\delta$ as their inner boundary}
  \label{fig:unitspeed}
\end{figure}

Consider the setting as depicted in figure \ref{fig:unitspeed} and recall the notation from the
proof of lemma \ref{preannulus}. In addition to the hyperbolic annular cover and the uniformization
of $B$, we now need suitable charts for a cylinder $C''$ around $\delta$, in which the natural
parametrization has an easy description.

In the situation where $\gamma$ is disjoint from $\lambda$, one can simply use the hyperbolic collar
neighbourhood of $\gamma$ on $X$ as $C''$ and obtain charts by biholomorphically embedding this as a
round annulus into $\CC$ (also see the proof of lemma \ref{preannulus}).

If $\gamma=\gamma_i$, one uses the \textit{extended grafting half-cylinder} as $C''$ -- that is the
annulus bounded by $\delta$ and the ``old'' boundary curves for the standard hyperbolic collar
around $\gamma$ on $X$.

To obtain charts, note that $C''$ is (projectively) of the form
$$C'' = S^\theta_t/\left<z \mapsto e^{l(\gamma)}\right>, \quad\quad S^\theta_t  = \{r\cdot e^{i\varphi}, r > 0, \pi/2-\theta \leq \varphi \leq \pi/2+\theta+t \} \subset \CC^*$$
where $\theta=\theta(l)$ is the angle corresponding to the standard hyperbolic collar on $X$.  (see
the proof of lemma \ref{deltabounding} and \cite{KD07} or \cite{McMullen:1998mz} for more details)
This cylinder carries a natural flat metric realizing it as $C'' = \gamma \times [-\theta,
t+\theta]$ such that $\gamma\times[0,t] \subset C''$ is exactly the natural flat metric on the
grafting cylinder.

It also has an embedding $C'' \to \CC$ (using the exponential function) such that closed
geodesics (of the flat metric) parametrized in unit speed correspond to round circles in $\CC$
parametrized in constant angular speed.

As $C''$ is an annulus on the surface with core curve homotopic to $\gamma'$, we
can \textit{biholomorphically} lift it into the hyperbolic annular cover (though not necessarily
into the collar) -- denote this lift by $C$.

The pre-annulus map $\phi^j$ sends the hyperbolic geodesic $\gamma'$ in natural parametrization to
$f(\delta)$ parametrized by constant speed (in the uniformizing chart) as it is the composition of a
scaling and a shearing along $\alpha^j$ (which fixes $f(\delta)$), then sends it back using
$f^{-1}$.

Thus, the distortion we have to compensate is the distortion of $F$ (cf. figure \ref{fig:unitspeed})
on the inner boundary circle, where $F$ is the composition of $f^{-1}$ and the inverse of the lift
map $C'' \to C$. Note that $F$ is only defined in some neighbourhood of $\delta$.
 
We would now like to postcompose the pre-annulus maps with shearing maps undoing the distortion of $F$.

However, to apply the estimate from proposition \ref{shearing} (shearing maps), the shear parameter
has to fix a point on the boundary circle.

To ensure this here, we have to apply a twist map of $B$ first, with a twisting amount of less than
1. Using the estimates for twist maps (proposition \ref{twistmaps}) and noting that $\Mod(B) \geq
T/l$ for some constant $T$ we see

$$\log K = \frac{2}{\sqrt{1+4\Mod(B)^2}-1} \leq \frac{2}{2\Mod(B) - 1} \leq \frac{2l}{2T-l}$$
which is less than a constant times $l$.

As the logarithims of quasiconformality constants behave additively under composition, we are done
once we show that the unshearing map satisfies the desired dilatation bound.
This however, follows using lemma \ref{shearing} and the following proposition.
\begin{prop}[distortion of $F$]
  \label{Fdistortion}
  The restriction of $F$ to the inner boundary circle (in above context) is $L$-bilipschitz with respect to the angular
  metric, where
  $$L \leq (\mathrm{const.})\cdot l^{1/4}+1$$
\end{prop}
\textit{Proof.}$\;$ Let $l = l_X(\gamma)$, $l' = l_{\gr_{\lambda}X}(\gamma')$. Denote by $A$ the standard hyperbolic
half-collar on $\gr_{\lambda}X$ around $\gamma'$. Then we have $\Mod(A) =
\theta(l(\gamma'))/l(\gamma')$.  Let $B$ (as before) be the annulus bounded by $\alpha^1$ and
$\delta$. By lemma \ref{deltamod} (moduli of bounding annulus) we know for the appropriate $R$
  $$\frac{\theta(l') + R}{l'} = \Mod(C^1) \geq \Mod(B) \geq \Mod(C^2) \geq \frac{\theta(l') - R}{l'}$$
  The modulus of the extended grafting cylinder can be obtained by just adding the modulus of the
  collar and the grafting cylinder (see the proof of lemma \ref{deltabounding}). Thus, in the case
  where $\gamma = \gamma_i$, we have
$$\Mod(C) = \Mod(C'') = \frac{t/2 + \theta(l)}{l}$$
If on the other hand $\gamma$ is disjoint from $\lambda$ and thus $C$ is the ``old'' collar
neighbourhood of $\gamma$, we have
$$\Mod(C) = \frac{\theta(l)}{l}$$
It is well-known that there is an universal constant $\kappa$, such that any annulus of modulus $M$
in $\CC$ contains a round subannulus of modulus $\geq M - \kappa$ (cf. e.g. \cite[Chapter
2]{McMullen:1994uo}) as long as $M$ is large enough (and as $l$ is small we can always assume this
here).

We now constuct another annulus $D$, distinguishing two cases: if $\alpha^1$ is inside $C$, we just
set $D = B$. Otherwise let $D$ be the maximal subannulus in $C$ having a round outer boundary and
$\delta$ as inner boundary.  In that case, we have $\Mod(D) \geq \Mod(C) - \kappa$. In both cases,
denote the image of $D$ under the uniformizing map by $D'$, the preimage under the lift-map by
$D''$.

The last two annuli we need are the corresponding maximal round subannuli: $E'' \subset D''$ and $G'
\subset D'$.
Now we are set to use lemma \ref{distortion} (controlling boundary distortion)

$F$ maps $G'$ holomorphically into $D$, then into $C''$. Thus on $\delta$ it is Lipschitz with dilatation
$$\frac{\Mod(C'')}{\Mod(G')} \leq \frac{\Mod(C)}{\Mod(D)-\kappa} \leq \frac{\Mod(C)}{\Mod(B)-\kappa} \quad\mathrm{or}\quad \frac{\Mod(C)}{\Mod(C)-2\kappa}$$
depending how $D$ was defined (see above). Similarly, the inverse mapping $F^{-1}$ sends $E''$ into
$B$, thus its Lipschitz constant on $\delta$ is
$$\frac{\Mod(B)}{\Mod(E'')} \leq \frac{\Mod(B)}{\Mod(D)-\kappa} \leq \frac{\Mod(B)}{\Mod(B)-\kappa} \quad\mathrm{or}\quad \frac{\Mod(B)}{\Mod(C)-2\kappa}$$
Obviously there are two types of expressions we have to estimate. Let us start with 
$$L = \frac{\Mod(C)}{\Mod(C)-2\kappa} = 1 + \frac{4\kappa}{\Mod(C)-2\kappa}$$
As $\Mod(C) \geq (\mathrm{const.})/l$, $L-1$ in this case is actually smaller than a constant times $l$. 
Similarly, as 
$$\Mod(B) \geq \frac{\theta(l')-R(l)}{l'} \geq \frac{\mathrm{const.}}{l}$$
we can handle the third expression.

For the other two, we use the estimates for $\theta$ and $R$ we have obtained before and $\theta(l) \leq \pi/2$. Consider first the case $\gamma = \gamma_i$. Then 
$$L = \frac{\Mod(C)}{\Mod(B)-\kappa} \leq \frac{ \frac{t/2+\theta(l)}{l}  }{ \frac{\theta(l')-R}{l'}  - \kappa}
= \frac{\frac{l'}{l}(t/2 + \theta(l))}{\theta(l')-R-\kappa l'}
\leq\frac{\frac{\pi}{\pi+t}\cdot\frac{t+\pi}{2}}{\frac{\pi}{2}-\frac{l'}{2}-R-\kappa l'} $$
Thus $L-1$ is smaller than
$$L-1 \leq \frac{(\frac{1}{2}+\kappa)l'+R}{\frac{1}{2}\pi -(\frac{1}{2}+\kappa)l' -R}$$

The nominator is smaller as $\mathrm{(const.)}\cdot l^{1/4}$, and the denominator is larger than a
constant (for small $l$). This yields the claim.
Next, consider the case where
$$L = \frac{\Mod(B)}{\Mod(C)-2\kappa} \leq \frac{ \frac{\theta(l')+R}{l'}  }{ \frac{t/2+\theta(l)}{l} -2\kappa  } 
\leq \frac{ \theta(l')+R }{ \frac{l'}{l}(t/2+\theta(l)) -2\kappa l'  }$$
Recall that (cf. lemma \ref{deltabounding})
$$\frac{l'}{l} \geq K_1\frac{2\theta}{2\theta+t},\quad\quad \frac{l'}{l}\frac{2\theta+t}{2} \geq K_1\theta \geq \frac{1}{1+l}\theta \geq \frac{1}{1+l}\cdot\frac{\pi-l}{2}$$
Thus we can further estimate
$$L \leq  \frac{(\theta + R)(1+l)}{\frac{\pi-l}{2} -2\kappa l'(1+l)}
\leq \frac{\frac{\pi}{2} + R + l\theta + l\cdot R}{\frac{\pi}{2} - \frac{l}{2} - 2\kappa l'(1+l)}$$
So at the end, we obtain the estimate
$$L-1 \leq \frac{R + l\theta + l\cdot R + l/2 + 2\kappa l'(1+l)}{\frac{\pi}{2} - \frac{l}{2} - 2\kappa l'(1+l)}$$
Again, the denominator is larger than some constant, while the nominator can be estimated as less
than a contant times $l^{1/4}$.

It remains to show the estimates in the case where $\gamma$ is disjoint from $\lambda$. Here we have
$$L = \frac{\Mod(C)}{\Mod(B)-\kappa} \leq \frac{ \frac{\theta(l)}{l}  }{ \frac{\theta(l')-R}{l'}  - \kappa} = \frac{\frac{l'}{l}\theta}{\theta-R-l'\kappa} \leq \frac{\theta}{\theta-R-l'\kappa}$$
and thus 
$$L-1 \leq \frac{R+l'\kappa}{\theta-R-l'\kappa}$$
Again, the denominator is larger than some constant, while the nominator can be estimated by a constant times $l^{1/4}$. The final case is 
$$L = \frac{\Mod(B)}{\Mod(C)-2\kappa} \leq \frac{ \frac{\theta(l')+R}{l'}  }{ \frac{\theta(l)}{l} -2\kappa  } 
\leq \frac{\frac{l}{l'}(\theta(l')+R) }{ \theta(l) -2\kappa l' } \leq \frac{\frac{1}{K}(\theta(l')+R) }{ \theta(l) -2\kappa l' } \leq  \frac{(1+l)(\frac{\pi}{2}+R) }{ \frac{\pi-l}{2} -2\kappa l' }$$
Thus, 
$$L-1 \leq \frac{\pi l + R + R\cdot l + 2\kappa l'}{\frac{\pi-l}{2} - 2\kappa l'}$$
which as above yields the claim. This finishes the proof of proposition \ref{Fdistortion} and lemma \ref{annulusmaps}.
\end{proof}

\subsection{Controlling the twist}
\label{sec:twist}
As a last step, we have to bound -- and compensate -- the twist induced by the annulus
maps $\Phi^j$.
Twist may be introduced by two different sources.

First, there is the twist created by our basic maps: scaling does not create any twist, and shearing
induces twists of amount strictly less than one.
As we use a definite, finite number of basic maps to construct the annulus maps, we do not care
about these twists -- compare the proof of lemma \ref{annulusmaps} to see that the error bounds for
an untwisting of a fixed amount is of the right magnitude.

The other source for twist is the uniformizing map $f$ of the annulus $B$ (recall the construction in
lemma \ref{preannulus} (pre-annulus maps) and compare figure \ref{fig:twist})
\begin{figure}[htbp!]
  \centering
  \includegraphics[width=\textwidth]{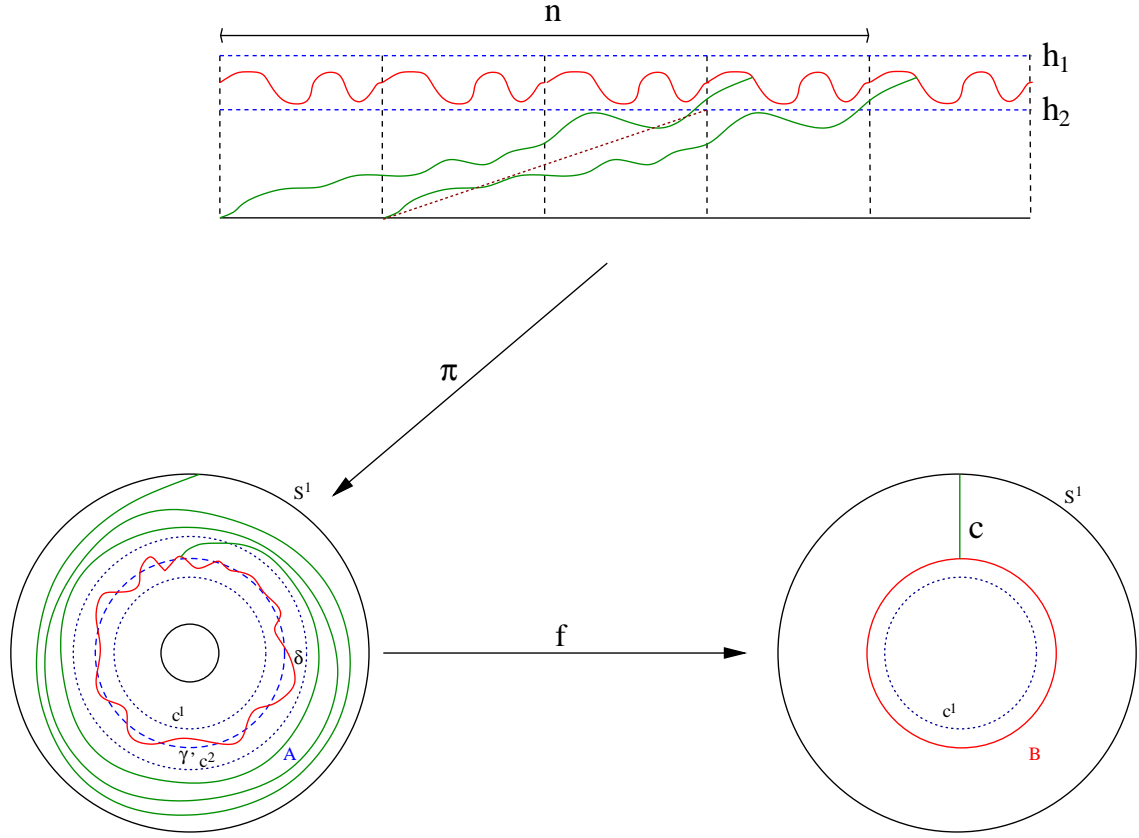}
  \caption{Controlling the twist of the uniformizing map $f$}
  \label{fig:twist}
\end{figure}
To estimate this twist, we use logarithmic coordinates for the annulus $A$. 
Consider a straight arc $c$ in the (uniformized) annulus $B$ (right hand side of figure \ref{fig:twist}).
The fact that $f$ induces a twist of amount $n$ is equivalent to saying that lift of $f^{-1}$
traverses $n$ fundamental domains.

We can use this formulation to prove
\begin{lemma}[Twist bound for uniformizing map]
  \label{twistbound}
  In the context of the proof of lemma \ref{preannulus}, the map $f$ induces a twist of less than $n$, where
  $$(n-2)^2 \leq \Mod(C^1)^2 - \Mod(C^2)^2$$
\end{lemma}
\begin{proof}
  Consider the quadrilateral $Q$ bounded by $c$, $S^1$ and $f(\delta)$ in $B$. As $c$ is a straight
  arc in the round annulus $B$, we can label the sides of $Q$ such that
  $$\Mod(Q) = \Mod(B)$$
  As $f$ is biholomorphic on $B$, the inverse image $Q' = f^{-1}(Q)$ is a quadrilateral of the same
  modulus. Furthermore, the modulus of $Q$ (and $Q'$) is the extremal length of the family $\Gamma$
  of curves connecting the top and bottom side of $Q$ (or $Q'$) with respect to the chosen
  labelling. Thus
  $$\mathrm{ext}(\Gamma) = \Mod(B) \leq \Mod(C_1).$$

  Now consider the metric on $Q'$ induced by the euclidean metric on $C^1$ (that is, in logarithmic coordinates).
  Any arc in $Q'$ connecting the bottom and top side than has to traverse at least $n-2$ horizontal
  segments (of length 1), and at least a height of $h_2$. Thus, the euclidean length $l$ of any such
  arc satisfies
  $$l^2 \geq (n-2)^2 + h_2^2$$
  The area of $B$ in this metric satisfies
  $$\mathrm{area}(B) \leq h_1\cdot 1$$
  Thus, by definition, we know that the extremal length of the family $\Gamma$  satisfies
  $$\mathrm{ext}(\Gamma) \geq \frac{(n-2)^2 + h_2^2}{h_1}$$
  Using the estimate above, we see
  $$\frac{(n-2)^2 + h_2^2}{h_1} \leq \Mod(C^1)$$
  and, recalling that $h_i = \Mod(C^i)$,
  $$(n-2)^2 \leq \Mod(C^1)^2 - \Mod(C^2)^2$$
\end{proof}
So we need to compensate a twist of at most
$$\sqrt{\Mod(C^1)^2 - \Mod(C^2)^2}+2$$
Again, a constant number of twists yields an error which is of the right magnitude ($\leq
C\sqrt{l}$), so we only have to worry about the first part.
To do so, we first estimate
$$\left(\frac{\Mod(B)}{n} \right)^2 \geq \frac{\Mod(C^2)^2}{\Mod(C^1)^2 - \Mod(C^2)^2} 
= \frac{1}{L-1}$$ where $L = (\Mod(C^1)/\Mod(C^2))^2$. Now we plug this into proposition
\ref{twistmaps} (twist maps), to obtain that the quasiconformality constant of the untwisting map
satisfies
$$\log(K) \leq \frac{2}{\sqrt{1+4\frac{1}{L-1}}-1}$$
Using lemma \ref{deltamod} and the estimate  $R \leq T\cdot l^{1/4}$ for the appropriate $R$ (and some constant $T$), as well as $\theta(l) \geq \pi/2-l/2$ we compute
$$L \leq \left(\frac{\theta(l')+R}{\theta(l')-R} \right)^2 \leq 
\left(\frac{\pi/2-l'/2+T\cdot l^{1/4}}{\pi/2-l'/2-T\cdot l^{1/4}}\right)^2
=\left(1+ \frac{2T\cdot l^{1/4}}{\pi/2- l'/2 - T\cdot l^{1/4}} \right)^2$$
Furthermore, for small $l$, we can estimate this to be (for some contants $k, k'$)
$$\leq (1+k\cdot l^{1/4})^2 \leq 1 + k'\cdot l^{1/4}.$$
Thus,
$$\sqrt{1+4\frac{1}{L-1}} \geq \sqrt{1+\frac{4}{k'\cdot l^{1/4}}} \geq \frac{2}{\sqrt{k'}\cdot l^{1/8}}$$
This gives
$$\frac{2}{\sqrt{1+4\frac{1}{L-1}}-1} \leq \frac{2}{\frac{2}{\sqrt{k'}\cdot l^{1/8}}-1}
=  \frac{2\cdot l^{1/8}}{\frac{2}{\sqrt{k'}}-l^{1/8}} \leq C\cdot l^{1/8}.$$
Summarizing this calculation, we see
\begin{lemma}[Quasiconformality of untwist map]
  The quasiconformality constant $K$ of the untwist map satisfies
  $$\log K \leq C\cdot l^{1/8}$$
  for small $l$ and a universal constant $C$.
\end{lemma}

\begin{lemma}[comparison maps]
  \label{compmaps}
  There is a comparison map $\Psi: X \to X'$ satisfying
  \begin{enumerate}[i)]
  \item $\Psi$ sends $\gamma'$ in natural parametrization to $\delta$ in natural parametrization.
  \item $\Psi$ preserves the marking on $X'$.
  \item $\Psi$ is quasiconformal with dilatation $K$, and
    $$\log K \leq C \cdot l_{X}(\gamma)^{1/8}$$
    for some universal constant $C$.
  \end{enumerate}
\end{lemma}
\begin{proof}
  We compose the annulus maps from section \ref{sec:amaps} with the necessary untwist maps to obtain property \textit{ii)}. By the lemma above, property \textit{iii)} is satisfied.
\end{proof}

\subsection{Finishing the proof}
\label{sec:finish}

We are now equipped to prove the theorems stated in section \ref{sec:poqft}.
\begin{proof}[Proof of the iteration theorem]
  Let us assume all $\gamma_i$ are short enough to apply the estimates of the preceding sections. This defines $\widetilde{\epsilon}$. The constant $C$ is given by the universal constant from corollary \ref{compmaps}.

   Consider the multicurve
$$\Gamma = \left(s_1\cdot \frac{l_X(\gamma_1)}{l_{\gr_{\lambda}X}(\gamma_1)}+t_1\right)\gamma_1 + \ldots + \left(s_n\cdot \frac{l_X(\gamma_n)}{l_{\gr_{\lambda}X}(\gamma_n)}+t_n\right)\gamma_n.$$
  Note that $\gr_{\Gamma}X$ can be obtained from $X'$ by inserting a flat cylinder of circumference $l = l_X(\gamma_i)$ and
  height $\frac{l_X(\gamma_i)}{l_{X'}(\gamma_i)}\cdot s_i$ at the flat core curve $\delta_i$ of the already inserted grafting cylinder (see section \ref{sec:amaps}).

 On the other hand, $\gr_{\eta}X'$ is obtained from $X'$ by cutting at $\gamma'_i$ and gluing in a flat annulus of height $s_i$ and circumference $l_{X'}(\gamma'_i)$. 
  By the construction of $\Gamma$, the moduli of the glued in cylinders are equal.

 Now we use the comparison maps constructed in lemma \ref{compmaps}. As all the constructions took place in the 
 hyperbolic collar around $\gamma_i'$, the comparison maps for all $\gamma_i'$ can be combined to a single map
 $\Phi:X \to X'$, which has the properties stated in lemma \ref{compmaps} for each $\gamma_i$ simultaneously.

 Because of property \textit{i)} of the comparison maps and the fact that the moduli of the glued in cylinders are equal they extend to a map
  $$\gr_{\eta}(\gr_{\lambda}X) \to \gr_{\Gamma}X$$
  of the same quasiconformality constant -- which is of the right magnitude (property \textit{iii)}). As the marking is preserved (property \textit{ii)}), this yields the desired bound on Teichmüller distance.

  It remains to show that $\gr_{\Gamma}X$ and $\gr_{\eta\widetilde{+}\lambda}X$ are close to each other. As both are obtained 
  as a grafting on the same support, it is enough to compare the heights of the respective grafting cylinders.
  But by lemma \ref{deltabounding} we know
  $$\frac{\pi + t_i}{\pi} \leq \frac{l_X(\gamma_i)}{l_{\gr_{\lambda}X}(\gamma_i)} \leq K^{-1}\frac{2\theta + t_i}{2\theta}.$$
  Thus, we need to estimate
  $$\frac{s_i\cdot K^{-1}\frac{2\theta + t_i}{2\theta} + t_i}{s_i\cdot \frac{\pi + t_i}{\pi} + t_i}
  = \frac{\frac{s_i}{t_i}\cdot K^{-1}\frac{2\theta + t_i}{2\theta} + 1}{\frac{s_i}{t_i}\cdot \frac{\pi + t_i}{\pi} + 1}
  \leq \frac{K^{-1}\frac{2\theta + t_i}{2\theta}}{\frac{\pi + t_i}{\pi}}$$
  as $(1+x)/(1+y) \leq x/y$ for $x\geq y$. But is is just
  $$K^{-1} \frac{\pi}{2\theta} \frac{2\theta + t_i}{\pi + t_i} \leq K^{-1} \frac{\pi}{2\theta}
  \leq (1+l_i)\frac{\pi}{\pi-l_i} \leq 1 + \mathrm{const}\cdot l$$
  Thus, the logarithm of the quasiconformality constant of the rescaling map is smaller than some constant times
  $l$ -- from which follows the claim.
\end{proof}

\begin{proof}[Proof of the seperation theorem]
Using the argument of the proof above, we find that
$$d_\teich(\gr_\eta\gr_\lambda X, \gr_{\Gamma}X) \leq C\cdot\left(\max_{i=1,\ldots, m}l_X(\gamma_i) \right)^{1/8}$$
where
$$\Gamma = \eta + \frac{l_X(\gamma_{n+1})}{l_{\gr_\lambda X}(\gamma_{n+1})}\cdot t_{n+1}\gamma_{n+1} + \ldots + \frac{l_X(\gamma_m)}{l_{\gr_\lambda X}(\gamma_m)}\cdot t_m\gamma_m$$
However, by lemma \ref{deltabounding} \textit{iii)} we know that 
$$1 \leq \frac{l_X(\gamma_{n+i})}{l_{\gr_\eta X}(\gamma_{n+i})} \leq \frac{1}{K_1} \leq 1+l_X(\gamma_{n+i})$$
Hence, if we build a quasiconformal map $\gr_{\gamma}X \to \gr_{\eta+\lambda}X$ as above
which just rescales the cylinders from height $\frac{l_X(\gamma_{n+i})}{l_{\gr_\lambda
    X}(\gamma_{n+i})}\cdot t_{n+i}$ to height $t_{n+i}$, its quasiconformality constant will be less
than $1+\max(l_X(\gamma_{n+1}))$. This yields the claim.
\end{proof}

\section{Holonomy lifts and grafting rays}
\label{sec:holonomylifts}

We now turn to the results on holonomy lifts of Teichmüller space sketched in the introduction. 
Recall that the slices $\teich_\lambda(S)$ for integral $\lambda$ are exactly the projective 
structures having Fuchsian holonomy by Goldman's theorem. Also note that there are two natural
parametrizations of the slice $\teich_\lambda$: one can use the geometric coordinate $\Gr_\lambda(X) = Z \mapsto X$
and the conformal coordinate $\Gr_\lambda(X) = Z \mapsto \pi(Z)$.
The difference between these coordinates is measured by the conformal grafting map $\gr_\lambda:\teich(S) \to \teich(S)$.
\begin{theorem}[Holonomy lifts of grafting rays]
  \label{sec:holon-lifts-graft:main1}
  Let $X$ be a hyperbolic surface and $\lambda$ be a short integral lamination on $X$ (i.e. all curves are shorter than the universal constant from theorem \ref{quasiflow}).
  \begin{enumerate}[i)]
  \item There is a $r>0$, such that for each $n$ the holonomy lift 
    $$g_n(s) = \gr_{n\lambda}\left( \gr_{s\lambda}X \right)$$
    is contained in the $r$-tube around the grafting ray $s \mapsto \gr_{s\lambda}X$.
   \item There is a $R>0$ such that the following holds. Let $\eta$ 
     be an short integral lamination, disjoint from $\lambda$. Then the holonomy lifts
     $$g_\eta(s) = \gr_{\eta}\left(\gr_{s\lambda}X\right)$$
     are contained in the $R$-tube around the grafting ray $\gr_{s\lambda}(\gr_{\eta}X)$ for each $\eta$ as above.
  \end{enumerate}
\end{theorem}
\begin{proof}
  The first statement follows by using theorem \ref{quasiflow}. In fact, once the length
  of $\lambda$ is shorter than $\widetilde{\epsilon}$, we know that
  $$d_\teich(\gr_{n\lambda}\left( \gr_{s\lambda}X \right), \gr_{(n\lambda)\widetilde{+}(s\lambda)}X) \leq C\cdot \max\left(l_X(\gamma)\right)^{1/8}$$
  To see that $\gr_{(n\lambda)\widetilde{+}(s\lambda)}X$ is close to the grafting ray in direction $\lambda$, let 
  $\lambda = t_1\gamma_1 + \ldots + t_m\gamma_m$ and recall the definition
  $$(n\lambda)\widetilde{+}(s\lambda) = \left(n\cdot t_1\cdot \frac{\pi + t_1\cdot s}{\pi} + s\cdot t_1\right)\gamma_1 + \ldots \left(n\cdot t_m\cdot \frac{\pi + t_m\cdot s}{\pi} + s\cdot t_m\right)\gamma_m$$
  Thus, the weight of the curve $\gamma_i$ is
  $$\left(n\cdot\frac{\pi + t_i\cdot s}{\pi} + s\right)t_i$$
  We now want to rescale the cylinders as in the proof of theorem \ref{quasiflow}. To this end, let $t=\min_i(t_i)$ and define
  $f(s) = n\cdot\frac{\pi+t\cdot s}{\pi} + s$. Then the quotients of the heigths satisfy
  $$\frac{n\cdot\frac{\pi + t_i\cdot s}{\pi} + s}{f(s)} = 
  \frac{n\cdot\frac{\pi + t_i\cdot s}{\pi} + s}{n\cdot\frac{\pi+t\cdot s}{\pi} + s} \leq \frac{\pi + s\cdot t_i}{\pi + s\cdot t}\leq \frac{t_i}{t}$$
  and therefore the quotients of the heights can be estimated from the weights of $\lambda$ alone.
  This however implies, that there is a quasiconformal map
  $$\gr_{(n\lambda)\widetilde{+}(s\lambda)}X \;\to\; \gr_{f(s)\lambda}X$$
  whose dilatation is bounded by $\log\frac{\max_it_i}{\min_it_i}$. This shows the first claim.

  The second claim follows by simply applying theorem \ref{quasicomm} to the situation in \textit{ii)}.
\end{proof}
Using this theorem we obtain the statement about Teichmüller geodesics mentioned in the introduction
\begin{cor}[Holonomy lifts of Teichmüller geodesics]
    There is a $\epsilon>0$ such that the following holds. Let $\delta>0$ and an integral lamination $\lambda$ be given. Consider the set $\,\mathcal{U} \subset \teich(S)$ of all hyperbolic surfaces on which $\lambda$ has length less than $\epsilon$ and each simple closed curve disjoint from $\lambda$ has length at least $\delta$.

  Then there is a $r>0$, such that for each $X \in \mathcal{U}$ and $n \in \NN$ the holonomy lift 
    $$g_n(s) = \gr_{n\lambda}\left( \rho_{\lambda,X}(s) \right)$$
of the Teichmüller geodesic $\rho_{\lambda,X}$ through $X$ in direction $\lambda$ is contained in the $r$-tube around the geodesic $\rho$.
\end{cor}
\begin{proof}
  Let $X \in \mathcal{U}$ be given. We first show that there is a $r = r_X > 0$ that fulfills the claim each ray starting in $X$.
  Using the triangle inequality it suffices to bound
  $$d_\teich\left(\gr_{n\lambda}\rho_{\lambda,X}(s), \gr_{n\lambda}\gr_{s'\lambda}(X) \right)+d_\teich\left( \gr_{n\lambda}\gr_{s'\lambda}(X), \gr_{t'\lambda}(X) \right)+d_\teich\left(\gr_{t'\lambda}(X), \rho_{\lambda,X}(t) \right)$$
By a theorem of Diaz and Kim \cite[theorem 4.3]{KD07} there is a $R$ such that for each $t$ there is a $t'$ with
$$d_\teich(\gr_{t'\lambda}X, \rho_{\lambda, X}(t)) < R$$
By the preceding theorem, the middle term is bounded by some uniform constant (if $s'$ is chosen appropriate to $t'$). 

Thus, it remains to estimate the first term. To do so, we use Minsky's product region theorem \cite[theorem 6.1]{Minsky:1996xe}. Note that along the grafting ray $t \mapsto \gr_{t\lambda}Y$ the length of each curve disjoint from $\lambda$ has length bounded from below and above.
This is due to the fact that grafting along $\lambda$ decreases the length of curves disjoint from $\lambda$ (see \cite[theorem 3.1]{McMullen:1998mz}). By the collar lemma this also implies that no curve disjoint from $\lambda$ can become too short. The same is true for the Teichmüller geodesic $\rho_{\lambda,X}$. 
Hence, the projections $\Pi_0(\gr_{n\lambda}\rho_{\lambda,X}(s))$ and $\Pi_0(\gr_{n\lambda}\gr_{s'\lambda}X)$ (in the notation of \cite[theorem 6.1]{Minsky:1996xe}) are contained in a compact subset of $\teich(S\setminus\lambda)$.

It remains to show that $\Pi_i(\gr_{n\lambda}\rho_{\lambda,X}(s))$ and $\Pi_i(\gr_{n\lambda}\gr_{s'\lambda}X)$ are close in $\HH^2$.
Using the theorem of Diaz and Kim again, we see that this is the case for $\Pi_i(\rho_{\lambda,X}(s))$ and $\Pi_i(\gr_{s'\lambda}X)$.
However, using lemma \ref{deltabounding} we see that the quotient 
$$\frac{l_{\gr_{n\lambda}\rho_{\lambda,X}(s)}(\gamma_i)}{l_{\gr_{n\lambda}\gr_{s'\lambda}(X)}(\gamma_i)}$$
of the lengths of a curve $\gamma_i \subset \lambda$ on $\gr_{n\lambda}\gr_{s'\lambda}X$ and $\gr_{n\lambda}\rho_{\lambda,X}(s)$ is bounded.

On the other hand, lemma \ref{twistbound} (or, alternatively, \cite[proposition 3.5]{KD07}) implies that the product of the length of $\gamma_i$ and the twist around $\gamma_i$ on $\gr_{n\lambda}Y$ is bounded independent of $n$ if $\lambda$ is short on $Y$. 
Therefore, the projections $\Pi_i(\gr_{n\lambda}\rho_{\lambda,X}(s))$ and $\Pi_i(\gr_{n\lambda}\gr_{t\lambda}X)$ stay bounded distance apart in $\HH^2$ for all $n$.
Now the product region theorem implies the claim.

As all estimates above depend continuously on the geometry of $\Pi_0(X)$, for any compact $K \subset \teich(S\setminus\lambda)$ there is a constant $r$ such that the claim is fulfilled for any $Y \in \mathcal{U}$ with $\Pi_0(Y) \in K$.
As the mapping class group of $S\setminus\lambda$ acts cocompactly on $\Pi_0(\mathcal{U})$ the desired statement follows.
\end{proof}

To prove a theorem concerning grafting rays through holonomy lifts, we first introduce a convenient notation for iterated grafting
$$\gr_{\lambda}^nX :=  \underbrace{\gr_{\lambda}(\gr_{\lambda}( \dots \gr_{\lambda}}_{n \, \mbox{times}}(X)\dots).$$

\begin{theorem}[Iterated holonomy lifts]
  \label{hollifts}
  Fix a closed oriented surface $S$ and a simple closed curve $\gamma$. Let $X$ be a hyperbolic structure on $S$, 
  such that $\gamma$ is shorter than $\widetilde{\epsilon}$. Consider the grafting ray
  $$g_0(t) = \gr_{t\gamma}X$$
  and the iterated holonomy lifts
  $$g_{n,m}(t) = \gr_{2\pi m\gamma}^n(\gr_{t\gamma}X)$$
  Then for any $n,m \in \NN$, $g_{n,m}$ is contained in the $r$-tube around $g_0$, where 
  $$r = C\cdot (l_X(\gamma))^{1/8}.$$
\end{theorem}
\begin{proof}
  Define
  $$l_{n,m} := l_{\gr_{2\pi m\gamma}^nX}(\gamma)$$
  As each $2\pi m$-grafting decreases the length of $\gamma$ by at least a factor of $1/3$ 
  (compare lemma \ref{deltabounding} \textit{i)}) we have
  $$l_{n,m} \leq \left(\frac{1}{3}\right)^nl_X(\gamma)$$
  Using theorem \ref{quasiflow} we obtain 
  $$d_\teich(g_{n+2,m}(t), \gr_{(2\pi m\gamma)\widetilde{+}(2\pi m\gamma)}g_{n,m}(t)) \leq C\cdot (l_{n,m})^{1/8}$$
  However, $(2\pi m\gamma) \widetilde{+} (2\pi m\gamma) = S_1\gamma$ for some $S_1$.
  Therefore, we can iterate
  $$d_\teich(\gr_{S_1\gamma}g_{n,m}(t), \gr_{(S_1\gamma)\widetilde{+}(2\pi m\gamma)}g_{n-1,m}(t)) \leq C\cdot (l_{n-1,m})^{1/8}.$$
  By repeating this estimate and combining it with the inequality for $l_{n,m}$ quoted above, we see
  $$d_\teich(g_{n+2,m}(t), \gr_{L(n,m,t)\gamma}X) \leq C\cdot(l_X(\gamma))^{1/8}\cdot\sum_{k=0}^n\left(\frac{1}{3^{1/8}}\right)^k$$
  for some $L(n,m,t)$. But, as the geometric series is converging, the sum on the right hand side of the 
  inequality is uniformly bounded in $n$ and the theorem follows.
\end{proof}
\begin{theorem}
   Let $X$ be any hyperbolic structure on $S$ and let $m \in \NN$ be given. Then there is a $R$ such that for 
   any $n \in \NN$, $g_{n,m}$ is contained in the $R$-tube around $g_0$. 
\end{theorem}
\begin{proof}
  Once $K$ is big enough, theorem \ref{hollifts} yields that $g_{K+n,m}$ is contained in a $r$-tube around $g_{K,m}$. 

  But $d_\teich(X, \gr_{2\pi m\gamma}X) < K(l,m)$ for some $K(l,m)$ if the length of $\gamma$ on $X$ is
  smaller than $l$ (collapse the extended grafting cylinder to the old collar, see \cite[proof of proposition 3.4]{KD07} 
  for more details on this argument). As the length of $\gamma$ decreases along the
  grafting ray (lemma \ref{deltabounding}), this gives that $g_{K,m}$ is contained in some tube around $g_0$ by the 
  triangle inequality.
\end{proof}
   
\begin{cor}[holonomy lifts follow Teichmüller geodesic ray]
     The iterated holonomy lifts $g_{n,m}$ of a grafting sequence are contained in a tube around the Teichmüller geodesic through $X$ defined by $\gamma$.
\end{cor}
\begin{proof}
     This follows from theorem \ref{hollifts} and the fact that grafting rays
     are bounded distance apart from the corresponding Teichmüller geodesic (see \cite[theorem 4.3]{KD07})
\end{proof}
If we do not consider the holonomy lifts of a grafting ray, but instead grafting rays through holonomy lifts of the starting point, we get the following stronger result
\begin{theorem}[Grafting rays through holonomy lifts accumulate]
  \label{accumulate}
  Let $X$ be any hyperbolic surface and $\gamma$ a simple closed geodesic on $X$. Consider the grafting rays
  $$c_{n,m}(t) = \gr_{t\gamma}(\gr_{2\pi m\gamma}^nX).$$
  For large values of $n$, the $c_{n,m}$ accumulate exponentially fast
  $$d_\teich(c_{n+1,m}(t), c_{n,m}(2\pi + a_{n,m}t)) \leq C\cdot q^n$$
  for some $0<q<1, a_{n,m} > 1$ and a constant $C$ depending on $X$. In particular, these rays accumulate in the Hausdorff topology on Teichmüller space.
\end{theorem}
Looking at this result, one might hope that also the holonomy lifts of grafting rays as in above theorem actually accumulate in the Hausdorff topology of Teichmüller space. However, the methods developed in this paper seem unsuitable to prove this. 
\begin{proof}
  The proof is very similar to the preceding one. We use the same notation. Once $l_{n,m}$ is small enough, theorem \ref{quasiflow} yields
  $$d_\teich(c_{n+1,m}(t), c_{n,m}(S_{n,m}t + 2\pi m)) \leq C\cdot (l_{n-1,m})^{1/8}$$
  Using the estimate 
  $$l_{n,m} \leq \left(\frac{1}{3}\right)^nl_X(\gamma)$$
  we then see the claim for short $l_{n,m}$. Furthermore the same estimate also gives that once $n$ is large 
  enough, $l_{n,m}$ will be arbitrary short, and the claim follows.
\end{proof}
We turn now to the case where we replace the curve $\gamma$ in above iteration theorems by a general integral lamination $\lambda = 2\pi n_1 \gamma_1 + \ldots + 2\pi n_m\gamma_m$.
Here, in general the statements will be false.

To see this, consider a typical example in the setting of theorem 5.5. 
Let $X$ be a hyperbolic surface, and $\gamma_1, \gamma_2$ be two curves on the surface having equal length. 
Let $\lambda = \pi \gamma_1 + 2\pi\gamma_2$. 
If we graft along $\lambda$, by lemma \ref{deltabounding} we see for the lengths
$$K\frac{2\theta}{2\theta + \pi} \leq \frac{l_{\gr_\lambda X}(\gamma_1)}{l_X(\gamma_1)} \leq \frac{\pi}{\pi + \pi},
\quad\quad K\frac{2\theta}{2\theta + 2\pi} \leq \frac{l_{\gr_\lambda X}(\gamma_2)}{l_X(\gamma_2)} \leq \frac{\pi}{\pi + 2\pi}$$
So, in each grafting step, the length of $\gamma_1$ decreases roughly by a factor of $1/2$, while the length of $\gamma_2$ gets scaled by $1/3$. Thus the quotient of the lengths $l_{\gr^n_\lambda X}(\gamma_2)/l_{\gr^n_\lambda X}(\gamma_1)$ will tend to 0.

On the grafting ray corresponding to $\lambda$ on the other hand, this length quotient will be bounded away from 0
(again, using lemma \ref{deltabounding}) and thus, by Wolpert's theorem, the rays through the holonomy lifts will move further and further away from the grafting ray.

\section{Geometric convergence of grafting rays and sequences}
\label{sec:geom_conv}

As another application of the methods developed above, we want to study geometric limits of grafting rays and sequences. 
Let us first precisely define what we mean by geometric convergence.
\begin{defi}[geometric convergence]
  Let $Z$ be a marked oriented Riemann surface of genus $g$ with $2n$ cusps. 
  We say that a sequence $X^i$ of closed marked oriented Riemann surfaces \textit{converges geometrically} to $Z$ if:

  For any $\epsilon>0$, and any collection of neighbourhoods $U_1,\ldots, U_{2n} \subset Z$ of the
  cusps, biholomorphic to the punctured unit disk $\Delta^*$, there is a number $N>0$ such that for all
  $k > N$ there are simple closed curves $\gamma_1^k, \ldots, \gamma_n^k$ on $X^k$ and a marked
  (orientation preserving) homeomorphism
  $$F_k : X^k \setminus\{\gamma_1^k, \ldots, \gamma_n^k\} \to Z$$
  (where the marking on $X^k \setminus\{\gamma_1^k, \ldots, \gamma_n^k\}$ is the one induced by the
  marking on $X^k$) such that $F^{-1}$ restricted to $Z\setminus (U_1 \union \ldots \union U_{2n})$ is
  $(1+\epsilon)$-quasiconformal.
\end{defi}

Given a hyperbolic surface $X$ and a weighted multicurve $\lambda = t_1\gamma_1 + \ldots + t_n\gamma_n$, we now construct a candidate  $\gr_{\infty\cdot \lambda}X$ for the ``endpoint'' of the grafting ray

\begin{defi}[Endpoint of grafting ray]
\label{endpoints}
Cut $X$ at $\lambda$ to obtain a hyperbolic surface $Z$ with $2n$ boundary curves $\gamma^1_i, \gamma^2_i$. 
Take $2n$ punctured disks with boundary
$$\overline{\Delta}^*_{i,j} = \{0<z\leq 1\} \subset \CC,\quad i=1,\ldots, n, \, j=1,2$$
and glue $S^1 \subset \overline{\Delta}_{i,j}^*$ (in unit euclidean speed) to $\gamma_i^j$ on $Z$ (in constant hyperbolic speed).
We call the resulting punctured Riemann surface $\gr_{\infty\cdot\lambda}$ the \textit{endpoint of the grafting ray}.  
\end{defi}
\begin{lemma}[grafting rays converge]
  \label{grafting_rays_converge}
  Let $X \in \teich(S)$ be a Riemann surface, $\lambda$ a weighted multicurve on $X$. Then the grafting ray $s \mapsto \gr_{s\cdot \lambda}X$
  converges geometrically to $\gr_{\infty\cdot \lambda}X$ as $s\to\infty$.
\end{lemma}
\begin{proof}
  We need to show, that for any collection of neighbourhoods $U_{i,j}$ of the cusps and for sufficiently large $t$ there is a map 
  $$F_s : \gr_{s\lambda}X \setminus\{\gamma_1^k, \ldots, \gamma_n^k\} \to \gr_{\infty\cdot\lambda}X$$
  which is $(1+\epsilon)$-quasiconformal outside the $U_{i,j}$.

  Denote the glued in punctured discs on $\gr_{\infty\cdot\lambda}X$ by $\Delta^*_{i,j}$ as in definition \ref{endpoints} and fix biholomorphic charts $\{z\in \CC, 0<z<1\} \to \Delta^*_{i,j}$.
 Let $A_{\delta} = \{z \in \CC, 0<z<\delta\}$ be the radius-$\delta$ punctured discs in these charts. 

 As any neighbourhood of a cusp contains (the image of) $A_\delta$ for sufficiently small $\delta$, it suffices to 
 constuct the maps $F_t$ in the case where all $U_{i,j}$ are of the form $A_\delta$.

 We now consider the situation around one of the $\gamma_i$. 
 Once $s$ is large enough to ensure that the modulus of the grafting cylinder $C_{\gamma_i}$ at $\gamma_i$ 
 is larger than the modulus of both of the ``remaining parts'' $\Delta_{i,j}^*\setminus A_\delta = \{z\in\CC, \delta<z<1\}$
 for j=1,2 -- that is 
 $$\Mod(C_{\gamma_i}) = \frac{s}{l_X(\gamma_i)} \geq 2 \Mod(\Delta\setminus A_\delta)$$
 we can constuct a homeomorphism of $C_{\gamma_i} \setminus \gamma_i$ to $\Delta_{i,1}^*\union\Delta_{i,2}^*$ which is conformal outside $U_{i,1} \union U_{i,2}$.

 To do so, decompose the grafting cylinder into three round annuli
 $$C_{\gamma_i} = \gamma_i \times [0,s\cdot t_i] = C_1 \union C_2 \union C_3, \quad\quad C_i = \gamma \times [a_i, b_i]$$
 such that $\Mod(C_1) = \Mod(C_2) = \Mod(\Delta\setminus A_\delta)$. 
 
Now we can map $C_1$ conformally to $\Delta_{i,1}\setminus U_{i,1}$, and $C_3$ conformally to $\Delta_{i,2}\setminus U_{i,2}$. Choosing any homeomorphism $C_2\setminus\delta_i \to U_{i,1}\union U_{i,2}$ we obtain the desired map.
As both $\gr_{\infty\cdot\lambda}X$ and $\gr_{s\lambda}X$ are obtained by surgeries at $\gamma_i$, we can combine these maps with the identity on the complement of the grafting cylinders on $\gr_{s\lambda}X$ and obtain  
$$F_s : \gr_{s\lambda}X \setminus\delta_i \to \gr_{\infty\cdot\lambda}X$$
such that $F_s^{-1}$ is conformal outside the $U_{i,j}$.
\end{proof}
Note that the endpoint does not depend on the weigths on $\lambda$, but just its support.

\vspace{3mm}
We now want to prove a similar result for iterated grafting sequences. Let $X \in \teich(S)$ be a base point and choose a weighted multicurve $\lambda = t_1\gamma_1 + \ldots + t_n\gamma_n$ on $X$. Define the sequence
$$X^{m} = \gr_{\lambda}^mX$$
Denote by $\gamma^m_i$  the (hyperbolic) simple closed geodesic on $X^m$ in the free homotopy class of $\gamma_i$ and by $\delta_i^m$ the flat core curve of the grafting cylinder around $\gamma_i$ on $X^m$. 

As a first step we need to understand how the endpoints of the $\lambda$-grafting rays through the $X^m$ behave.
\begin{prop}[comparing endpoints of grafting rays]
  \label{raycompare}
  Let $X^m$ be the iterated grafting sequence defined above.

  Then the Teichmüller distance of the endpoints of the grafting rays through the terms of the sequence decreases exponentially:
  $$d_\teich(\gr_{\infty\cdot\lambda}X^m, \gr_{\infty\cdot\lambda}X^{m+1}) \leq C\cdot q^m$$
  for some $C>0, 0<q<1$.
\end{prop}
 
\begin{proof}
 First we note that any two punctured disks $\{0<z\leq1\}$ and $\{0<z\leq r\}$ are biholomorphic.
 Thus, cutting $X^m$ at $\gamma_i^m$ and glueing in two puncured disks yields the same surface as glueing in a cylinder of length $s\cdot t_i$ first and then glue punctured disks to the core curve of that cylinder.

 In other words, the endpoint $\gr_{\infty\cdot\lambda}X^m$ is biholomorphic to the surface obtained by cutting $\gr_{\lambda}X^m = X^{m+1}$ at $\delta_i^m$ and glueing punctured disks to the boundary components.

 Now, using the comparison maps as in the proof of theorem \ref{quasiflow}, we see that
  $$d_\teich(\gr_{\infty\cdot\lambda}X^m, \gr_{\infty\cdot\lambda}X^{m+1}) \leq C\cdot \left(\max_{i=1,\ldots, n}l_{X^m}(\gamma_i)\right)^{1/8}$$
As the length of $\gamma_i^m$ on $X^m$ decreases exponentially (lemma \ref{deltabounding}), we conclude the proposition.
\end{proof}
\begin{cor}[Endpoints converge]
 \label{cauchy}
  The sequence $\gr_{\infty\cdot\lambda}X^m$ is a Cauchy sequence in the Teichmüller space of $S\setminus\bigcup\gamma_i$.
\end{cor}
As the Teichmüller space of a surface of finite type is complete, we see that the sequence of endpoints $\gr_{\infty\cdot\lambda}X^m$ has a limit $X^\infty$.
Now we are equipped to show the desired convergence theorem.
\begin{theorem}[Geometric convergence of grafting sequence]
  The $\lambda$-grafting sequence $X^m$ converges geometrically to $X^\infty$.
\end{theorem}

\begin{proof}
  Pick any neighbourhoods $U_{i,j}$ of the cusps on $X^\infty$ and let $\epsilon > 0$ be given.

  By the preceding corollary \ref{cauchy}, there is a $M$ such that $\gr_{\infty\cdot\lambda}X^m$ has a distance small 
  enough to $X^\infty$, such that there is quasiconformal homeomorphism $f^m: \gr_{\infty\cdot\lambda}X^m \to X^\infty$ 
  with dilatation smaller than $\sqrt{1+\epsilon}$ for all $m > M$.

  Using theorem \ref{quasiflow} as in the proof of theorem \ref{hollifts}, we see that for large 
  $K>k$ there is a quasiconformal map 
  $$g^{K,k}:\gr_{\lambda}^KX \to \gr_{\widehat{\lambda}}\gr_{\lambda}^kX$$
  whose quasiconformality constant converges to 1, as $k \to \infty$.
  Here, $\widehat{\lambda}$ is some weighted multicurve (the $K-k$ times iterated weighted sum of $\lambda$) 
  with the same supporting curves as $\lambda$ and weights which are unbounded in $K$.

  Pick $k$ large enough such that this quasiconformality constant is also less than $\sqrt{1+\epsilon}$.
  Now fix $f=f^k: \gr_{\infty\cdot\lambda}X^k \to X^\infty$ and set $V_{i,j} = f^{-1}(U_{i,j})$. 
  Recalling the proof of lemma \ref{grafting_rays_converge}, there are weights $R_i$, such that once $r_i > R_i$ we have
  $$F_{r_1\gamma_1+\ldots+r_n\gamma_n}:\gr_{r_1\gamma_1+\ldots+r_n\gamma_n}X^N \to \gr_{\infty\cdot\lambda}X^N$$
  such that $F^{-1}$ is conformal on the complement of $\bigcup V_{i,j}$

  Now we choose $K$ large enough to ensure that the weights on $\widehat{\lambda}$ are larger than the critical $R_i$. 
  By composing $g^{n,k}$ with $F_{\widehat{\lambda}}$ for $n > K$ we then obtain a map $X^n \to X^\infty$ which has the desired properties.
\end{proof}

\vspace{20 mm}
\bibliographystyle{math}
\bibliography{../../biblio}

\bigskip

\noindent
Sebastian W. Hensel\\
Mathematisches Institut der Universität Bonn\\
Beringstraße 1\\
D-53115 Bonn\\

\smallskip

\noindent
e-mail: loplop@math.uni-bonn.de

\end{document}